\documentclass[preprint, 11pt]{article}
\usepackage[utf8]{inputenc}
\usepackage[margin=1in]{geometry}  
\usepackage{amsthm,amssymb,amsmath}
\usepackage[square,numbers]{natbib}
\usepackage{authblk,color}
\newtheorem{theorem}{Theorem}
\usepackage{enumerate}
\usepackage{graphics}
\usepackage{graphicx}

\usepackage[colorinlistoftodos,prependcaption,textsize=tiny]{todonotes}
\setuptodonotes{fancyline, color=green!20}

\usepackage[hidelinks]{hyperref}

\usepackage{setspace}

\usepackage[makeroom]{cancel}
\usepackage[normalem]{ulem} %to strike the words

\newtheorem{proposition}[theorem]{Proposition}
\newtheorem{lemma}[theorem]{Lemma}
\newtheorem{prop}[theorem]{Proposition}

\newtheorem{remark}{Remark}
\newtheorem{example}{Example}

\usepackage{bbm}

\title{Joint Pricing and Innovation Control in Regulated Recycling-Rate Diffusion}
\author[1]{\small Bowen Xie
\thanks{bxie@uakron.edu}}
\author[2]{\small Yijin Gao
\thanks{Corresponding author, yjgaomath@163.com}}
\affil[1]{\footnotesize Department of Mathematics, 
College of Engineering and Polymer Science, 
The University of Akron, 
Akron, OH 44325-4002}
\affil[2]{\footnotesize School of Economics and Finance, Shanghai International Studies University,
Shanghai, 201620, PR China}
\date{}

\begin{document}

\maketitle

\begin{abstract}
% \todo{Revise this to include RL}
We introduce a regulated stochastic diffusion model for the recycling rate and formulate a joint control problem over production and process innovation via the dynamics of recycling investment and product pricing. 
The resulting stochastic control problem captures the system manager's trade-off between product-price decisions and investment expenditures under an infinite-horizon discounted cost structure. 
Owing to the recycling-rate specification, we incorporate two regulated state processes, which induce additional policy-driven cost components in the value function consistent with green-economy regulations. 
We resolve the jointly regulated stochastic production and process-innovation admission control problem by introducing the associated Hamilton-Jacobi-Bellman (HJB) equation and providing rigorous proofs that establish the correspondence between the HJB solution and the value function of the underlying control problem. 
The HJB equation is analyzed under mild, practically motivated assumptions on the system parameters. We further present numerical experiments and sensitivity analyses to illustrate the tractability of the HJB characterization and to assess the practical relevance of the imposed parameter conditions.

\vspace{0in}
\vspace{0in}

% 60K40, 

% \bigskip
\end{abstract}

% \newpage

%%%%% Table of contents %%%%%
\tableofcontents 

%%%%%%%%%%%%%%%%%%%%%%%%%%%%%%%%%%%%%%%%%%%%%%%%%%%%%%%%%%%%%%%%%%%%%%%%%%%%%%
\section{Introduction}
Production innovation and process innovation are two parallel paths for a firm to improve quality and reduce costs under green economy policies. 
It is well known that pricing and innovation policies depend simultaneously on the dynamics of demand and supply. 
A central question in this context is how firms evaluate the trade-offs associated with investing in these two types of innovation.
A substantial body of literature has examined the theoretical and empirical dimensions of product and process innovation. 
For example, \cite{fontana2009product} analyzed the product innovation and survival in a high-tech industry.
\cite{lay2012innovation} discussed the adoption and impact of new manufacturing concepts in the German industrial sector. 
\cite{miravete2006innovation} explored innovation complementarities and production scale. 
Additional studies on product innovation include \cite{fritsch2001product}, \cite{kirner2009innovation}, \cite{enkel2010creative}, and \cite{lin2013market}. 
Similarly, process innovation has also garnered significant attention across various industries.  
\cite{bauer2013exploration} examined process innovation in the chemical industry, while \cite{lager2010managing} addressed the management of process innovation from idea generation to implementation. In the context of Industry 4.0 technologies, green process innovation plays a pivotal role (see \cite{liu2019green}, \cite{de2021process}). \cite{ornaghi2006spillovers} provided evidence of spillovers in process innovation among manufacturing firms. Related contributions include \cite{bergfors2009product}, \cite{frishammar2012antecedents}, and \cite{chirumalla2021building}. 
Both product and process innovation are closely tied to a firm’s profitability. System managers can model these relationships through optimization frameworks that incorporate relevant constraints and trade-offs.

% \textcolor{blue}{explain the extra parts: $L(t), U(t)$}

In this article, we investigate a stochastic model of production and process innovation that explicitly incorporates the dynamics of recycling investment and product pricing, with a focus on the recycling rate within the framework of a circular economy. The state process concerned in this paper is the recycling rate. 
In contrast to the deterministic models presented in \cite{sethi1983deterministic} and \cite{sethi2008optimal}, where the recycling rate is restricted to the interval $[0, 1]$, our stochastic framework introduces a diffusion term that propagates throughout the system. This necessitates additional constraints to ensure the recycling rate remains within the admissible bounds. 
To address this, we introduce a regulated stochastic recycling process, whose sample paths are constrained by two local time processes, $L(t)$ and $U(t)$, which activate only when the recycling rate $r(t)$ reaches the lower boundary (0) or the upper boundary (1), respectively, for $t \in [0, T]$.

Importantly, our model preserves the structural formulation of the recycling rate dynamics found in deterministic models. 
Based on this regulated stochastic process, the system manager aims to maximize profit. We define a cost functional comprising two components: (i) profit generated from meeting market demand, and (ii) penalties incurred when the recycling rate falls to zero.
The integration of recycled components into production can reduce costs and influence demand, particularly when modeled using functions such as the Cobb-Douglas demand function. The profit function is thus formulated in terms of product price, recycling investment, and the recycling rate (see Section~\ref{section2}). Penalties are introduced to reflect government policies that promote environmental sustainability and carbon-neutral production practices.
Consequently, we formulate a stochastic control problem to analyze how the recycling rate affects production profitability, allowing for joint optimization of product pricing and investment in recycling.

Our main strategy for solving the stochastic control problem is to utilize the corresponding Hamilton-Jacobi-Bellman (HJB) equation. 
This approach has been widely adopted in prior research (cf. \cite{ghosh2010optimal}, \cite{weerasinghe2018controlling}, \cite{xie2024long}, \cite{budhiraja2024ergodic}, \cite{xie2025single}). Analyzing the HJB equation enables us to characterize the solution profile, offering a clear and concise representation of the control problem and its optimal strategy.
However, this problem poses significant challenges due to the complexity of constructing the optimal policy under a given cost structure and the intricacies of jointly optimizing the control policies. 

To this end, we establish a rigorous connection between our joint stochastic control problem and the solution to its associated HJB equation. A key challenge lies in demonstrating how various parameters influence the system. 
We observe that different parameter choices can lead to distinct optimization formulations and corresponding HJB equations. To accommodate multiple scenarios simultaneously, we introduce functional extensions within the HJB framework (see Remark \ref{remark: parameter a_1 <= 1} in Section \ref{sec: hjb}).

To address the extended HJB equation, we reformulate it as a boundary value problem with free initial conditions, which admits a similar structure to free-boundary value problems. 
Our theoretical analysis focuses on the autonomous case, which can be further reduced to a non-linear ordinary differential equation (ODE) under appropriate conditions. 
Several propositions are presented to elucidate the structure of the solution in the context of production and process innovation.
As discussed in Section~\ref{section4}, we consider a parameterized differential equation subject to an imposed initial condition. 
By varying this initial condition, we identify the optimal parameter that yields the admission control policy maximizing the objective profit functional. Following this analytical framework, we provide numerical examples in Section~\ref{section5} to illustrate the optimality of the control strategy and the behavior of the optimal state process.

Moreover, it is worth noting that we impose certain constraints on the product price function $p(\cdot)$ in our formulation. 
First, we assume that the product price is bounded above by a constant $p_0$, i.e., $p < p_0$. This reflects the practical consideration that while product managers, acting as rational economic agents, may adjust prices flexibly, they typically do not exceed a certain threshold due to market norms and consumer expectations.
Second, we require that the maximum allowable price satisfies $p_0 \geq c_v$, where $c_v > 0$ denotes the unit production cost of virgin resources. 
This assumption is grounded in the observation that the cost of virgin resources generally exceeds that of recycled resources. If $p_0 < c_v$, it would imply that the system manager is willing to accept a negative profit margin per unit, which may only be justified under abnormal pricing strategies such as predatory dumping.
To maintain the generality and practicality of our model, we adopt these standard assumptions and exclude such atypical pricing behaviors from our analysis.

% \textcolor{blue}{Our main contributions (novel aspects) are two-folds: (1) firstly, we add the $L(t),U(t)$......(2) Another }
%%%%%% Contribution %%%%%%%%
\subsection{Contributions}

To the best of our knowledge, this is the first study to focus on the formulation and corresponding HJB equation in production and innovation models in conjunction with the concept of propagation of Gaussian chaos. Our main contributions can be summarized as follows: 

(i) We propose a regulated diffusion process for the recycling rate, incorporating Gaussian external shocks into its dynamics. 
This stochastic behavior propagates through the system, influencing the balance between product and process innovation. 
To ensure the recycling rate remains within the physically meaningful interval $[0, 1]$, we introduce two regulation processes (see Remark \ref{rm: L and U}), which are often called the local time processes or reflection barriers. 
These constraints are essential due to the interpretation of the recycling rate as a bounded state variable. 

(ii) The formulation of the joint stochastic control problem incorporates a penalty term that activates when the recycling rate reaches zero. 
This reflects the growing societal emphasis on carbon-neutral production, which is widely recognized for its environmental and economic benefits. 
The penalization serves as an incentive mechanism to promote higher recycling rates. In practice, the penalty parameter within the cost/profit functional can be calibrated using empirical data from specific companies or industry sectors, and adjusted to reflect governmental policies and sustainability goals.

(iii) We provide a detailed derivation of the joint stochastic admission control problem and its corresponding HJB equation in Section~\ref{section4}. 
The implementation of this approach in our context is non-trivial due to the potential extensions arising from the selection of system parameters. 
In particular, we observe that the sensitivity of demand to price must be treated in two distinct cases: $a_1 \geq 1$ and $0 <a_1 < 1$, as these regimes lead to different structural properties in the profit function. 
To ensure completeness, we incorporate a functional extension that accommodates both cases in our analysis (see Remark~\ref{remark: parameter a_1 <= 1}). As a result, we do not impose any restrictive assumptions on the price sensitivity parameter, allowing our model to capture a broader range of market behaviors and demand responses. 

(iv) In the numerical simulation based on the formal HJB equation, we present the solution profile and conduct sensitivity analysis for system parameters. 
We observe that the sensitivity of demand to the price parameter $a_1$ plays a crucial role in the joint control problem, which aims to determine the optimal recycling investment and product price. 
In addition, solving the parameterized differential equation derived from the HJB is computationally intensive, particularly due to the difficulty of identifying optimal parameters through exhaustive search. 

% To investigate alternative, non-conventional approaches, we apply the RL algorithms, modeling the dynamic joint control problem as a Markov Decision Process (MDP). 
% While it is natural to seek methods that circumvent the complexities inherent in joint stochastic control and HJB formulations, the performance and parameter-tuning strategies of RL-based methods remain uncertain and warrant further investigation. 

% We explore reinforcement learning (RL) as an alternative approach to solving the joint stochastic optimization problem. 
% We aim to evaluate the performance and applicability of RL within this specific context. To benchmark against analytical solutions obtained from theoretical analysis, we conduct numerical simulations based on the formal HJB equation. However, solving the parameterized differential equation derived from the HJB is computationally intensive, particularly due to the difficulty of identifying optimal parameters through exhaustive search. 
% To investigate alternative, non-conventional approaches, we apply the RL algorithms, modeling the dynamic joint control problem as a Markov Decision Process (MDP). While it is natural to seek methods that circumvent the complexities inherent in joint stochastic control and HJB formulations, the performance and parameter-tuning strategies of RL-based methods remain uncertain and warrant further investigation. 

% \textcolor{blue}{Literature review: solver for optimal problem}
%%%%%%% literature review %%%%%%
\subsection{Related literature}

In management science, production and process innovation models have received considerable attention in the circular economy literature. 
\cite{bonanno1998intensity} examined the competitive dynamics and strategic choice between these two types of innovation. \cite{adner2001demand} explored the relationship between demand heterogeneity and technological evolution. \cite{helmes2013optimal} addressed an optimal advertising and pricing problem within a class of general new-product adoption models. 
In the context of sustainable supply chains, \cite{ghosh2015supply} and \cite{chen2019reverse} focused on the development and sensitivity analysis of green supply chains. \cite{chenavaz2017analytical} proposed an analytical framework linking product quality and advertising strategies. \cite{zhu2017green} investigated green product design in competitive supply chain environments. 
Furthermore, the interplay between product and process innovation has been studied in the context of Stackelberg competition between North-country firm and South-country firm (cf. \cite{chichilnisky1994north}, \cite{bonanno1998intensity}, \cite{wang2019product}). 

Moreover, control problems within the circular economy framework have been actively studied in the literature. 
\cite{liu2015joint} formulated an inventory model for perishable foods, where demand depends on both price and quality, which decays continuously over time.
They established a joint dynamic pricing and investment strategy. \cite{li2017dynamic} proposed a dynamic control model for a multiproduct monopolist, incorporating product and process innovation with knowledge accumulation through learning-by-doing. \cite{li2020dynamic} extended this framework by considering reference quality in the dynamic control of product and process innovation. \cite{schlosser2021circular} introduced a deterministic model for joint dynamic pricing and recycling investment.
For further results on optimization problems in production and process innovation, see \cite{chenavaz2012dynamic}, \cite{dawid2015product}, \cite{pan2016dynamic}, and \cite{madani2017sustainable}. 
Admission control problems in stochastic systems have a long history, beginning with the seminal work of \cite{bather1966continuous}. These problems have been extensively studied across various domains, including mathematical finance, management science, and industrial engineering.
In mathematical finance, stochastic control problems arise in the context of target zones for exchange rates and central bank interventions (cf. \cite{krugman1991target}, \cite{bertola1992target}, \cite{miller1996optimal}, \cite{cadenillas1999optimal}). In high-frequency trading, market-making problems often involve stochastic optimization (cf. \cite{bertsimas1998optimal}, \cite{guo2017optimal}). In queueing systems and production-inventory models under heavy traffic, buffer-length control problems involve trade-offs between the cost of rejected customers due to full buffers and the cost of customer abandonment due to long wait times (cf. \cite{koccauga2010admission}, \cite{weerasinghe2013abandonment}, \cite{weerasinghe2016optimal}, \cite{biswas2017ergodic}, \cite{arapostathis2018infinite}, \cite{yang2020optimality}, \cite{xie2024long}). 

Stochastic models in production and innovation have also been explored. \cite{zhang2011stochastic} proposed a stochastic production planning model for firms facing seasonal demand and market uncertainty. \cite{filar2001two} introduced a two-factor production model, where one factor follows a continuous controlled diffusion process and the other a discrete controlled jump process. \cite{parlar1985stochastic} incorporated a compound Poisson cumulative demand and a chance constraint. \cite{lee2006stochastic} and \cite{tang2012stochastic} analyzed stochastic production frontier models with group-specific temporal variation in technical efficiency and non-linear cost structures. \cite{leung2004robust} developed a robust optimization model for stochastic aggregate production planning. 
Additional studies on stochastic innovation include \cite{allen1982some}, \cite{reynolds1992stochastic}, \cite{gutierrez2005forecasting}, \cite{tan2008estimating}, \cite{liu2009multi}, and \cite{coletti2016stochastic}.

% In recent years, reinforcement learning (RL) has gained significant traction as a data-driven approach for solving stochastic control problems, particularly in complex joint optimization settings. 
% For instance, \cite{walton2021learning} provided insights into the application of supervised, online, and reinforcement learning techniques to queueing systems. \cite{dai2022queueing} extended the theoretical framework of advanced policy gradient methods to Markov Decision Processes (MDPs) with infinite state spaces, unbounded costs, and long-run average objectives—features commonly encountered in complex queueing networks. \cite{quer2022connecting} explored the connection between stochastic optimal control and RL, demonstrating how optimal control problems can be framed within the RL paradigm. 
% Recent research has also focused on developing RL algorithms specifically tailored to stochastic control, often outperforming traditional heuristics across various traffic regimes. 
% For instance, \cite{jia2024online} proposed two online learning algorithms—Batch Upper Confidence Bound and Batch Thompson Sampling (BTS)—for price-based revenue management with reusable resources. Additional contributions by \cite{dai2021refined}, \cite{feng2021scalable}, \cite{chen2023online}, \cite{baron2024supervised}, and \cite{xie2025single} further illustrate the growing role of RL in addressing stochastic control challenges. 

% \textcolor{blue}{organization}
\subsection{Organization}

The remainder of this article is organized as follows. 
In Section~\ref{section2}, we present a basic deterministic model of production and process innovation, along with its associated control problem, as a motivating example. 
Section~\ref{section3} introduces the stochastic control problem based on a regulated stochastic recycling rate process. We begin by outlining the fundamental assumptions and then formulate the corresponding Hamilton-Jacobi-Bellman (HJB) equation.
Section~\ref{section4} is devoted to deriving the connection between the stochastic control problem and its associated HJB equation. 
We also establish the joint optimal strategy for product pricing and recycling investment. In Section~\ref{section5}, we provide numerical examples to illustrate the theoretical results and perform sensitivity analysis. 
% Section~\ref{sec6 RL} investigates the reinforcement learning (RL) approach for addressing the joint stochastic optimization problem and establishes a comparison between the theoretical results and data-driven methodologies. 
Finally, Section~\ref{section7} concludes the paper with a summary of findings and potential directions for future research.

\subsection{Notation}

% \textbf{Notations. }
Let $\mathbb{N}$ represent the set of positive integers. Let $\mathbb{R}$ denote the set of real numbers, and $\mathbb{R}_+^* = \{x\in\mathbb{R}: x > 0\}$ and $\mathbb{R}_+ = \{x\in\mathbb{R}: x\geq 0\}$. 
For $0< T\leq \infty$, let $\mathcal{D}[0, T]$ denote the Skorokhod space of functions with right continuous and left limit (RCLL or c\`adl\`ag) from $[0, T]$ to $\mathbb{R}$, equipped with the usual Skorokhod topology. 
The uniform norm on $[0, T]$ for a stochastic process $X$ in $\mathcal{D}[0, T]$ is defined by $\|X\|_T = \sup_{t\in[0, T]}|X(t)|$. 
% \begin{equation}
%     \|X\|_T = \sup_{t\in[0, T]}|X(t)|.  
% \end{equation}
% Throughout, we use $\Rightarrow$ to denote weak convergence in the Skorokhod space $D[0, T]$. 
For any real number $a$, $a^+ = \max\{a, 0\}$ and $a^- = \max\{-a, 0\}$. For any two real numbers $a$ and $b$, $a\wedge b = \min\{a, b\}$ and $a\vee b = \max\{a, b\}$.

%%%%%%%%%%%%%%%%%%%%%%%%%%%%%%%%%%%%%%%%%%%%%%%%%%%%%%%%%%%%%%%%%%%%%%%%%%%%%%
\section{Motivation}\label{section2}

As a motivating example, we consider a basic deterministic model of production and process innovation studies in the literature, along with its corresponding control problem, in the absence of propagation of chaos. This model captures the joint dynamics of pricing strategy and recycling investment. 
To this end, we introduce two variables: the investment expense in recycling $u(t)\geq0$ and the recycling rate $r(t)\in[0, 1]$ for $t\in[0, T]$, where $T>0$ is a constant. 
Many firms tend to invest in green processes by maintaining a recycling rate that contributes to environmentally friendly operations. 
However, fluctuations in the recycling rate may lead to degradation in investment effectiveness toward greenness. 
This relationship can be modeled by a differential equation: 
\[
    d r(t)=R(u(t),r(t))\, dt, 
\]
where $r(0) = r_0\in[0, 1]$ is a constant, and the recycling dynamics are characterized by a twice-continuously differentiable function $R:\mathbb{R}_+ \times [0, 1] \mapsto \mathbb{R}$. 

Additionally, firm managers are assumed to behave as \textit{Homo economicus}, aiming to maximize profit or minimize cost under appropriate constraints. 
It is therefore natural to formulate a control problem that captures the trade-off between environmental sustainability (greenness) and economic profitability. 
We consider a finite-horizon discounted cost functional and establish the following objective function (cf. \cite{schlosser2021circular}): 
\begin{equation}\label{iniwithoutbm}
    \begin{split}
       & V_1=\sup_{u(.),p(.)\geq 0} \int_0^T e^{-\alpha t} \pi(p(t),u(t),r(t))\, dt,\\
        & \text{subject to } \, d r(t)=R(u(t),r(t))\, dt, \quad r(0)=r_0, 
    \end{split}
\end{equation}
where $T > 0$ and $r_0\in[0, 1]$ are fixed constants, $p(\cdot)$ denotes the price of the product, $u(\cdot)$ represents the recycling investment made prior to demand realization, $r(\cdot)$ characterizes the recycling rate, and $\alpha>0$ denotes the discount factor. 
Here, the joint control variables are the investment expense in recycling $u(\cdot)$ and the product price $p(\cdot)$. 

More precisely, since the profit function $\pi(p(\cdot),u(\cdot),r(\cdot))$ has various definitions in the literature, we adopt the formulation proposed by \cite{schlosser2021circular} for clarity and consistency. 
In this framework, the profit function is defined as the revenue from sales minus the production and recycling investment costs: 
$$
\pi(p(t),u(t),r(t))=[p(t)-(1-r(t))c_v]D(p(t),r(t))-u(t),
$$
where $D:\mathbb{R}_+ \times [0, 1] \mapsto \mathbb{R}$  represents the customer demand, which depends on the product price $p$ and the recycling rate $r$, and $c_v > 0$ denotes the unit production cost of virgin resources. 
To align with the previous research, We account for the following assumptions of the demand function: 
\[
\frac{\partial D}{\partial p}<0, \frac{\partial D}{\partial r}\geq 0, \frac{\partial^2 D}{\partial p \partial r}\leq 0, 
\]
which reflect standard economic intuition such that demand decreases with price, increase with recycling rate, and the marginal effect of recycling on demand diminishes as price increases. 

For instance, \cite{li2020dynamic} and \cite{schlosser2021circular} adopt a linear demand form: 
\[
D(p(\cdot), q(\cdot), r(\cdot)) = a-a_1 p(\cdot)+a_2q(\cdot)+a_3[q(\cdot)-r(\cdot)],
\]
where $a$ is the baseline market potential, $a_1$ and $a_2$ capture the effects of the price $p(\cdot)$ and product quality $q(\cdot)$, respectively, and $a_3$ reflects the impact of the difference between product quality and reference quality. 
In our model, for simplicity and analytical tractability, we adopt the Cobb-Douglas demand function:
$
D(p,r)=a_0 p^{-a_1} r^{a_2}.
$

\begin{remark}

Further examples of profit functions used in the objective formulation \eqref{iniwithoutbm} can be found in the literature. 
For instance, \cite{chenavaz2012dynamic} proposed a simplified profit function:
$$
\pi(t)=(p-c)f-u_q(t)-u_c(t),
$$
where $p$ is the product price, $c$ is the unit production cost, $f$ denotes the demand, and $u_q(t)$ and $u_c(t)$ represent expenditures on product innovation and process innovation, respectively. 
In a more detailed setting, \cite{li2020dynamic} introduced a three-dimensional profit function that incorporates simultaneous product and process innovation: 
$$
\pi(t)=[p(t)-c(t)][a-a_1 p(t)+a_2 q(t)+a_3(q(t)-r(t))]-\left[\frac{\alpha}{2}k^2(t)+\frac{\theta}{2}h^2(t)+v q^2(t)\right].
$$
For brevity, we omit the detailed definitions of these variables and parameters; interested readers may refer to \cite{li2020dynamic} for a comprehensive exposition. 
\end{remark}

In practice, the evolution of the recycling rate $r(\cdot)$ can be influenced by various external factors, such as environmental protection policies, geographical conditions that facilitate the collection and storage of reusable materials, and the cost of reprocessing. 
To capture the inherent uncertainty in these influences, it is natural to incorporate random white noise into the dynamics of the recycling rate, representing small variations as potential external shocks.
This stochastic perturbation enables us to formulate a stochastic control problem that captures the joint dynamics of production and process innovation, while also allowing us to investigate the propagation of chaos within a circular economy framework. Accordingly, we propose a stochastic production and process innovation model in the next section.

\section{Stochastic model}\label{section3}

\subsection{Basic assumptions}\label{sec: assumptions}

In this section, we present the fundamental assumptions underlying the stochastic control problem. These assumptions serve as the basis for the model formulation introduced in the next section, where we establish the core dynamics and derive the formal HJB equation. The HJB equation can then be simplified and specified using the assumptions outlined here.

Our model is built upon a set of mild and standard assumptions. 
The recycling dynamics function $R:\mathbb{R}_+\times [0, 1]\mapsto \mathbb{R}$ is assumed to be a twice continuously differentiable function and satisfies the following conditions:
$$
\frac{\partial R}{\partial u}>0,\ \frac{\partial^2 R}{\partial u^2}<0, \frac{\partial R}{\partial r}\leq 1.
$$
A representative example of such a function is given by 
\begin{equation}
    R(u,r)=\gamma u^{\frac{1}{\gamma}} (1-r)-\delta r, 
\end{equation}
where $\gamma>1$ expresses the efficiency of recylcing investment, and $\delta>0$ is a proportional decay rate. 
This form captures diminishing returns in investment and the natural degradation of recycling effectiveness. 
Similar formulations can be found in \cite{sethi1983deterministic}, \cite{sethi2008optimal}, and \cite{schlosser2021circular}.

The choice of profit function $\pi:[0, \infty)\times[0, \infty)\times[0, 1]\mapsto\mathbb{R}$ is determined by three variable functions: $p(\cdot), u(\cdot)$, and $r(\cdot)$. 
We adopt the well-posed formulation proposed by \cite{chenavaz2012dynamic}, where the profit is defined as 
\begin{equation}
    \pi(t)=(p-c)f-u_q(t)-u_c(t),   
    \label{eq: profit function}
\end{equation}
where $u_q$ and $u_c$ represent expenditures on product and process innovation, respectively, and $c$ denotes the cumulative production cost. 
A related formulation, also inspired by \citet{chenavaz2012dynamic}, is given by
\begin{equation}
    \pi(p(t),u(t),r(t))=[p(t)-(1-r(t))c_v]D(p(t),r(t))-u(t),
    \label{eq: func of pi}
\end{equation}
where $p$ represents the retail price, $(1-r)c_v$ represents the effective unit cost of production, and $c_v> 0$ denotes unit cost of virgin resources. 
The demand function $D(p(\cdot), r(\cdot))$ follows the assumptions outlined in Section~\ref{section2}, where we employ the Cobb-Douglas demand function:
\begin{equation}
    D(p,r)=a_0 p^{-a_1} r^{a_2},
\end{equation}
where $a_0>0$ denotes the market potential, $a_1>0$ denotes the sensitivity of demand to price, $a_2>0$ denotes the sensitivity of demand to product greenness (cf. \cite{dai2017green}, \cite{saha2017optimal}, \cite{schlosser2021circular}).

\subsection{Joint stochastic control framework}
\label{sec: stochastic control problem}

In practice, the recycling rate $r(\cdot)$ in economic models may be influenced by external chaotic factors, which can be effectively modeled using Brownian motion. 
We define the state process of the recycling rate as a weak solution to the following stochastic differential equation:
\begin{equation}
    d r(t)=R(u(t),r(t))\, dt+\sigma\, dW(t), 
    \label{eq: sde}
\end{equation}
where $r(0) = r_0\in[0, 1]$, and $\sigma\in\mathbb{R}$ denotes the volatility parameter, $R:\mathbb{R}_+\times [0, 1]\mapsto \mathbb{R}$ is a twice continuously differentiable function, and $W(\cdot)$ denotes a one dimensional standard Brownian motion on a probability space $(\Omega, \mathcal{F}, P)$. 
The existence of a weak solution is well established in standard references on stochastic differential equations (cf. \cite{karatzas1991brownian}, \cite{protter2005stochastic}, \cite{oksendal2013stochastic}). 
Notice that since the recycling rate is constrained to the interval $[0, 1]$, additional construction of the state process is required to ensure this boundedness, which will be addressed shortly (also, see Remark \ref{rm: L and U}). 

Moreover, since firms are typically more concerned with the long-term behavior of their profits, we introduce the infinite-horizon discounted cost functional $J(r_0, u, p)$ as follows: 
\begin{equation}
    J(r_0, u, p) = E\left[ \int_0^{\infty} e^{-\alpha t} \pi(p(t),u(t),r(t))\, dt\right], 
    \label{eq: cost functional}
\end{equation}
where $\alpha>0$ is the discount factor, and $\pi:[0, \infty)\times[0, \infty)\times[0, 1]\mapsto\mathbb{R}$ denotes the profit function as defined in Section \ref{section2}. 
The objective of the joint stochastic control problem is to determine an optimal policy pair $(u^*, p^*)$ and the corresponding recycling rate process $r^*$ that jointly maximize the cost functional defined in \eqref{eq: cost functional}.

For a given $r_0\in[0, 1]$, we define the ordered triple $(r_{0}, u, p)$ as an admissible control system if the following conditions are satisfied 
\begin{enumerate}[(i)]
    \item $(r_{0}, u, p)$ is a weak solution to \eqref{eq: sde}, and
    \item the cost functional $J(r_0, u, p)$ is finite. 
\end{enumerate}
To define the value function, we introduce the collection of all the admissible controls, denoted by $\mathcal{A}(r_0)$. 
Notice that this set is nonempty since the cost functional is finite for zero controls. 
The value function associated with the joint stochastic control problem is then given by
\begin{equation}
    V_2(r_0) = \sup_{u, p\geq 0} J(r_0, u, p), 
    \label{equ: value function}
\end{equation}
for $r_0\in[0, 1]$. 

\begin{remark}\label{rm: L and U}
    Since the recycling rates are generally constrained to the interval $[0, 1]$, additional constraints may be necessary to guarantee their physical meaning. It is worth noting that the recycling model used in the deterministic optimization problem \eqref{iniwithoutbm} ensures that $r(t)\in[0, 1]$ for $t\in[0, \infty)$ provided the initial state satisfies $r_0\in[0, 1]$ (cf. \cite{sethi1983deterministic}, \cite{sethi2008optimal}). 
    However, this boundedness is not preserved when chaotic dynamics are introduced into the system (see Figure \ref{fig: sample graph of r without restrictions}). 
    \begin{figure}[h!]
        \centering
        \includegraphics[scale=0.5]{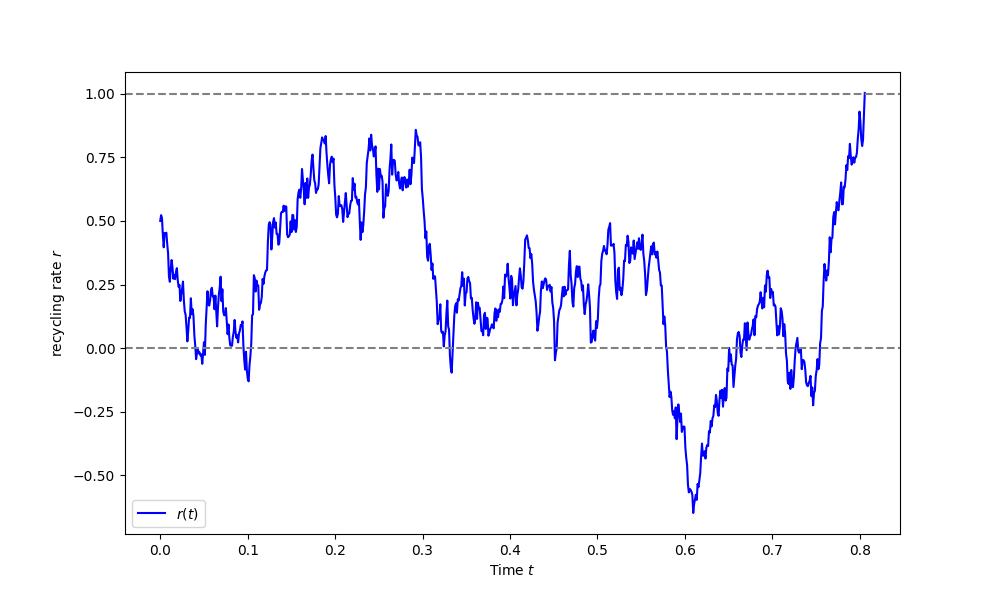}
        \caption{Sample path of stochastic differential equation \eqref{eq: sde}, where the $x$-axis denotes time $t$ and $y$-axis represents the recycling rate $r$.}
        \label{fig: sample graph of r without restrictions}
    \end{figure}
    Therefore, we introduce the following regulated stochastic differential equation:     
    \begin{equation}
        d r(t)=R(u(t),r(t))\, dt+\sigma \, dW(t)  + \,dL(t) - \,dU(t), 
        \label{eq: sde with reflections}
    \end{equation}
    where $L(\cdot)$ and $U(\cdot)$ are the regulation processes.
    More precisely, they satisfy that $L(0) = U(0) = 0$ and for any $T>0$, 
    \begin{equation}
    % \begin{aligned}
        \int_0^T \mathbbm{1}_{[r(s)>0]}(s)\,dL(s) = 0, \ \text{ and }
        \int_0^T \mathbbm{1}_{[r(s)<1]}(s)\,dU(s) =0. 
    % \end{aligned}
    \end{equation}
    Note that $L(t)$ and $U(t)$ are non-decreasing and increase only when $r(t)$ reaches the boundaries $0$ and $1$ for some time $t$, respectively. 
    These are referred to as local time processes in the literature (see, for instance, \cite{oksendal2013stochastic}). 
    The regulated stochastic differential equation guarantees that $r(\cdot)\in[0, 1]$ (see Figure \ref{fig: sample graph of regulated r} for a sample path). 
    Throughout this work, we consider \eqref{eq: sde with reflections} as the state process for the proposed joint stochastic control problem. 
    \begin{figure}[h!]
    \centering
        \includegraphics[scale=0.5]{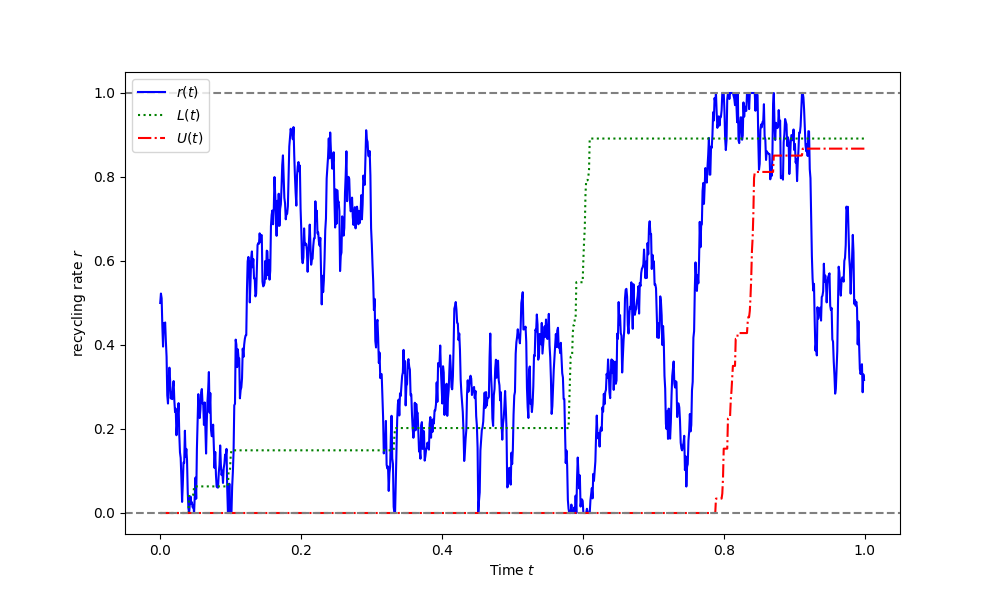}
        \caption{Sample path of regulated stochastic differential equation \eqref{eq: sde with reflections}, where the blue solid line exhibits the recycling rate $r$ with respect to time $t$, the red dashed line $U$ denotes the local time process of $r$ reaching the upper boundary 1, and the green dotted line $L$ represents the local time process of $r$ reaching the lower boundary zero. }
        \label{fig: sample graph of regulated r}
    \end{figure}
\end{remark}

Thus, the definitions of the state process \eqref{eq: sde}, cost/profit functional \eqref{eq: cost functional}, and value function \eqref{equ: value function} collectively yield a stochastic control formulation for the evolution equation under chaotic fluctuations:  
\begin{equation}\label{equwithbm}
    \begin{split}
        &V_2=\sup_{u(\cdot),p(\cdot)\geq 0}E\left[ \int_0^{\infty} e^{-\alpha t} \left(\pi(p(t),u(t),r(t))dt - C_L dL(t)\right)\right],\\
       &\text{subject to } d r(t)=R(u(t),r(t)) dt+\sigma dW(t)  + dL(t) - dU(t), 
    \end{split}
\end{equation}
% %%
% \textcolor{red}{*** Do we want to include a penalty for reaching the boundaries such that the cost functional would subtract $-C_L dL(t) - C_U dU(t)$? What are the physical meanings of this move? Also, let $C_U = 0$ since recycle rate 1 is encouraged ***}
% %%
where $r(0) = r_0\in[0, 1]$, and $\sigma\in\mathbb{R}$ denotes the diffusion coefficient, $R:\mathbb{R}_+\times [0, 1]\mapsto \mathbb{R}$ is a twice continuously differentiable function, $L(\cdot)$ and $U(\cdot)$ are the local time processes as defined above, and $W(\cdot)$ is a one dimensional standard Brownian motion on a probability space $(\Omega, \mathcal{F}, P)$. 
Moreover, the profit structure consists of two components: the revenue generated from the profit function $\pi$ as defined in Section \ref{section2}, which accumulates wealth at a rate of $\pi(p(t), u(t), r(t))\,dt$, and the penalty incurred when the recycle rate reaches zero, represented by a cost proportional to $dL(t)$ with a scalar weight $C_L > 0$. 
This penalization reflects the system manager's preference for maintaining a carbon-free environment.

\subsection{Hamilton-Jacobi-Bellman equation}
\label{sec: hjb}

Next, we introduce the formal Hamilton-Jacobi-Bellman (HJB) equation as the fundamental approach to solving the stochastic control problem \eqref{equwithbm}. 
One may follow the methodology proposed by \cite{pham2005some} to obtain the following HJB equation: 
% \begin{equation}\label{equhjb}
%     \begin{split}
%         \alpha V_2
%         % -\frac{\partial V_2}{\partial t}(t,r)
%         -\sup_{u,p \geq 0}\left\{\mathcal{L}_2  V_2(t,r)+\pi(p(t),u(t),r(t))\right\}=0, 
%     \end{split}
% \end{equation}
\begin{equation}\label{equhjb}
    \begin{split}
        \alpha Q
        % -\frac{\partial V_2}{\partial t}(t,r)
        -\sup_{u,p \geq 0}\left\{\mathcal{L}_2  Q(t,r)+\pi(p(t),u(t),r(t))\right\}=0, 
    \end{split}
\end{equation}
where $\mathcal{L}_2$ is the second-order infinitesimal generator associated with the regulated diffusion process \eqref{equwithbm}, given by
% $$
% \mathcal{L}_2 V_2=\frac{1}{2}\sigma^2 (V_2)_{rr} + R(u(t),r(t)) (V_2)_r.
% $$
$$
\mathcal{L}_2 Q=\frac{1}{2}\sigma^2 \frac{\partial^2 Q}{\partial r^2} + R(u(t),r(t)) \frac{\partial Q}{\partial r}.
$$

It is worth mentioning that more detailed derivations of the connection between the HJB and the joint stochastic control problem, using tools from stochastic analysis, are provided in later discussions (see Section~\ref{section4}). 
Our analysis bridges the gap between the solution to the stochastic control problem \eqref{equwithbm} and the solution to the HJB equation \eqref{equhjb}. 
Although this approach is considered standard in the literature, its implementation in our context is non-trivial due to potential extensions arising from the choice of system parameters. 

For brevity, we outline the strategy here and defer the technical details to Section~\ref{section4}. 
The strategy proceeds as follows: 
First, we establish a verification lemma that provides an upper bound for the value function defined in \eqref{equ: value function}. 
Second, we show that this upper bound coincides with the solution to the HJB equation \eqref{equhjb}; and finally, we demonstrate that this bound is attainable by selecting a specific control policy pair $(u^*, p^*)$, which is shown to be optimal. 
This strategy has been employed in previous works, including \cite{ghosh2010optimal}, \cite{weerasinghe2018controlling}, \cite{xie2024long}, \cite{budhiraja2024ergodic}, \cite{xie2025single}. It ensures consistency, allowing the stochastic control problem \eqref{equwithbm} to be reformulated as an HJB equation-solving problem.
Our objective is to investigate the HJB equation through both theoretical analysis and numerical implementation.

Under the assumptions in Section~\ref{sec: assumptions}, we derive the formal HJB equation as follows: 
% \begin{equation}
% \begin{aligned}
%     -\alpha V_2
%     % + \partial_t V_2
%     +\sup_{u,p\geq 0} \left\{\frac{1}{2}\sigma^2 V_2''  + V_2' [\gamma u^{1/\gamma} (1-r)-\delta r]+[p-(1-r)c_v]a_0 p^{-a_1}r^{a_2}-u \right\} &=0,\\
%     % \textcolor{red}{V_2'} &\geq 0, 
% \end{aligned}
% \label{equ: formal HJB}
% \end{equation}
\begin{equation}
\begin{aligned}
    -\alpha Q
    % + \partial_t V_2
    +\sup_{u,p\geq 0} \left\{\frac{1}{2}\sigma^2 Q''  + Q' [\gamma u^{1/\gamma} (1-r)-\delta r]+[p-(1-r)c_v]a_0 p^{-a_1}r^{a_2}-u \right\} &=0,\\
    % \textcolor{red}{V_2'} &\geq 0, 
\end{aligned}
\label{equ: formal HJB}
\end{equation}
with boundary conditions $Q'(0) = C_L$ and $Q'(1) = 0$. 
To solve this HJB equation, we observe that the supremum can be characterized by first-order optimality conditions, which further suggests 
% \textcolor{red}{Dealing with the $\max$ part, we find the first order condition of $u,p$:}
% \begin{align}
%     V_2' \gamma \frac{1}{\gamma} u^{1/\gamma-1} (1-r)-1=0,\\
%     a_0 r^{a_2} (1-a_1)p^{-a_1}-(1-r)c_v a_0 r^{a_2} (-a_1)p^{-a_1-1}=0.
% \end{align}
\begin{align}
    Q' \gamma \frac{1}{\gamma} u^{1/\gamma-1} (1-r)-1=0,\\
    a_0 r^{a_2} (1-a_1)p^{-a_1}-(1-r)c_v a_0 r^{a_2} (-a_1)p^{-a_1-1}=0.
\end{align}
We solve for $u$ and $p$ to obtain the explicit expressions:
% \begin{equation} 
%     u(t)=\left(\frac{1}{(1-r(t)) V_2'} \right)^{\frac{\gamma}{1-\gamma}}, \quad p(t)=\frac{a_1 c_v (1-r(t))}{a_1 - 1}.
%     \label{eq: equations for u and p}
% \end{equation}
\begin{equation} 
    u(t)=\left(\frac{1}{(1-r(t)) Q'} \right)^{\frac{\gamma}{1-\gamma}}, \quad p(t)=\frac{a_1 c_v (1-r(t))}{a_1 - 1}.
    \label{eq: equations for u and p}
\end{equation}
If we substitute these into the initial HJB equation and with some simple algebraic manipulations, we obtain the following non-linear second-order ordinary differential equation (ODE):
% \begin{equation}\label{eq: ode}
%     \begin{split}
%        \alpha V_2 
%        % - \partial_t V_2 
%        &= \frac{1}{2}\sigma^2 V_2'' + \gamma(1 - r)^{\frac{\gamma}{\gamma - 1}} (V_2')^{\frac{\gamma}{\gamma - 1}} - \delta r V_2' - (1 - r)^{\frac{\gamma}{\gamma - 1}} (V_2')^{\frac{\gamma}{\gamma - 1}}\\
%        &\quad+ a_0 r(t)^{a_2} ((1 - r)c_v)^{1 - a_1} \left[\left(\frac{a_1}{a_1 - 1}\right)^{1 - a_1} - \left(\frac{a_1}{a_1 - 1}\right)^{-a_1}\right].\\
%       %  &\left(\frac{1}{1-r} \right)^{\textcolor{red}{\frac{2-\gamma}{1-\gamma}}} (V_2)_r^{\frac{\gamma}{\gamma-1}} \gamma-\delta r\\
%       % &  +c_v^{1-a_1} a_0 (1-r)^{1-a_1} r^{a_2} \left[\left(\frac{a_1}{a_1-1} \right)^{1-a_1}-\left(\frac{a_1}{a_1-1} \right)^{-a_1}\right]-\left(\frac{1}{1-r} \right)^{\frac{\gamma}{1-\gamma}} (V_2)_r^{\frac{\gamma-1}{\gamma}} =0.
%     \end{split}
% \end{equation}
% \iffalse
% For simplicity, we assume $\gamma=1/2$ as the setting. Then $\frac{\gamma}{\gamma-1}=(\gamma-1)/\gamma=-1, \frac{2-\gamma}{1-\gamma}=3$
% \fi
% With some simple algebraic manipulations, we further rewrite the above equation as a second-order ODE:
% \begin{equation}\label{eq: ode}
%     \begin{split}
%         % \partial_t V_2+
%         \frac{1}{2}\sigma^2 V_2'' +
%         (\gamma - 1)(1 - r)^{\frac{\gamma}{\gamma - 1}} (V_2')^{\frac{\gamma}{\gamma - 1}} - \delta r V_2' 
%         +c(1-r)^{1-a_1} r^{a_2}-\alpha V_2=0, 
%     \end{split}
% \end{equation}
\begin{equation}\label{eq: ode}
    \begin{split}
        % \partial_t V_2+
        \frac{1}{2}\sigma^2 Q'' +
        (\gamma - 1)(1 - r)^{\frac{\gamma}{\gamma - 1}} (Q')^{\frac{\gamma}{\gamma - 1}} - \delta r Q' 
        +c(1-r)^{1-a_1} r^{a_2}-\alpha Q=0, 
    \end{split}
\end{equation}
where $c$ is a strictly positive constant defined as
\begin{equation}
    c = a_0c_v^{1 - a_1} \left[\left(\frac{a_1}{a_1 - 1}\right)^{1 - a_1} - \left(\frac{a_1}{a_1 - 1}\right)^{-a_1}\right]. 
    \label{eq: constant c}
\end{equation}
% To ensure $c\in \mathbb{R}_+$, we \textcolor{red}{assume $a_1 > 1$} for the rest of the discussion (see Remark \ref{remark: parameter a_1 <= 1} below). 
It is straightforward to verify that the solution to the non-linear ODE \eqref{eq: ode} also solves the HJB equation \eqref{equ: formal HJB}. 
Therefore, it suffices to consider the solution to the non-linear ODE, and we will further demonstrate that this solution also solves the original stochastic control problem \eqref{equwithbm}. 

\begin{remark}\label{remark: parameter a_1 <= 1}
    It is important to note that the condition $a_1 > 1$ is necessary to ensure that $c\in\mathbb{R}_+^*$. 
    When $0 < a_1 \leq 1$, we may consider the term involving $p$ in \eqref{equ: formal HJB}, namely $[p-(1-r)c_v]a_0 p^{-a_1}r^{a_2}$, becomes monotonically increasing in $p$. 
    In practice, the retail price $p$ is often subject to constraints, such as $p\leq p_0$, where $p_0 > 0$ is a reasonably large constant. 
    As a rational decision-maker (Homo economicus), the system manager seeks to maximize profit, even in the absence of recycling, by setting the highest permissible retail price. 
    Consequently, the price cap $p_0$ typically exceeds the unit production cost, i.e., $p_0 \geq c_v$, to ensure non-negative profit per unit without recycling. 
    Under these conditions, the supremum in \eqref{equ: formal HJB} is attained at the boundary $p = p_0$, and the first-order condition with respect to $p$ is no longer applicable due to the linearity of the objective in $p$. 
    Hence, \eqref{equ: formal HJB} and \eqref{eq: equations for u and p} imply the following modified ODE: 
    % \begin{equation*}
    %     \frac{1}{2}\sigma^2 V_2'' +
    %     (\gamma - 1)(1 - r)^{\frac{\gamma}{\gamma - 1}} (V_2')^{\frac{\gamma}{\gamma - 1}} - \delta r V_2' 
    %     +[p_0-(1-r)c_v]a_0 p_0^{-a_1}r^{a_2}-\alpha V_2=0.  
    %     % \label{eq: ode 0 < a_1 <= 1}
    % \end{equation*} 
    \begin{equation*}
        \frac{1}{2}\sigma^2 Q'' +
        (\gamma - 1)(1 - r)^{\frac{\gamma}{\gamma - 1}} (Q')^{\frac{\gamma}{\gamma - 1}} - \delta r Q' 
        +[p_0-(1-r)c_v]a_0 p_0^{-a_1}r^{a_2}-\alpha Q=0.  
        % \label{eq: ode 0 < a_1 <= 1}
    \end{equation*} 
    % However, in our discussion, we tend to consider a more general case so that we do not impose any constraint on the retail price $p \geq 0$. To this end, it suffices to assume the sensitivity of demand to price parameter $a_1 \geq 1$ for the rest of the discussion. 
\end{remark}

\begin{remark}
    In the formal HJB \eqref{equ: formal HJB}, we seek a solution that is non-decreasing, i.e., $Q'\geq 0$. 
    In the case of $Q'(x) < 0$ for some $x\in [0, 1]$, one observes that the expressions in \eqref{eq: equations for u and p} may not be well defined for certain values of $\gamma > 0$, particularly due to the presence of critical points at the boundaries in the lienar term. 
    It is straightforward to observe that the terms involving $u$, namely $Q'(x)\gamma u^{1/\gamma}(1 - x) - u$, have a negative derivative with respect to $u$, implying that the supremum is attained at the boundary, specifically at $u = 0$. Consequently, the HJB \eqref{equ: formal HJB} simplifies to 
    % \[
    % -\alpha V_2 + \sup_{p\geq 0} \left\{ \frac{\sigma^2}{2} V_2'' - \delta x V_2' + [p - (1 - x)c_v] a_0 p^{-a_1} x^{a_2} \right\} = 0. 
    % \]
    \[
    -\alpha Q + \sup_{p\geq 0} \left\{ \frac{\sigma^2}{2} Q'' - \delta x Q' + [p - (1 - x)c_v] a_0 p^{-a_1} x^{a_2} \right\} = 0. 
    \]
    Using the first-order condition for $p$ derived in \eqref{eq: equations for u and p}, we obtain the following second-order ODE for such values of $x$:
    % \begin{equation*}
    %     \frac{\sigma^2}{2} V_2'' - \delta x V_2' 
    %     +c(1-x)^{1-a_1} x^{a_2}-\alpha V_2=0, 
    %     % \label{eq: ode with V_2' < 0}
    % \end{equation*}
    \begin{equation*}
        \frac{\sigma^2}{2} Q'' - \delta x Q' 
        +c(1-x)^{1-a_1} x^{a_2}-\alpha Q=0, 
        % \label{eq: ode with V_2' < 0}
    \end{equation*}
    where $c>0$ is a constant defined in \eqref{eq: constant c}. 
\end{remark}

To address the discontinuity and accommodate a more general setting, we introduce an auxiliary function $F: \mathbb{R} \to \mathbb{R}_+$ defined by
\begin{equation}
    F(x) = 
    \begin{cases}
        x^{\frac{\gamma}{\gamma - 1}}, & \text{ if } x\geq 0, \\
        0, & \text{ if } x < 0, 
    \end{cases}
    \label{eq: function F}
\end{equation}
with its derivative given by
\begin{equation}
    F'(x) = 
    \begin{cases}
        \frac{\gamma}{\gamma - 1} x^{\frac{1}{\gamma - 1}}, & \text{ if } x\geq 0, \\
        0, & \text{ if }x < 0. 
    \end{cases}
    \label{eq: function F prime}
\end{equation}
% To better understand the extension, we provide some concrete examples of $\gamma$ and compute its Legendre transform. 
We further introduce an extension function $G: [0, 1] \times \mathbb{R}_+ \to \mathbb{R}_+$ to unify the treatment of the profit term across different regimes of $a_1$: 
\begin{equation}
    G(r, a_1) = 
    \begin{cases}
        c(1-r)^{1-a_1} r^{a_2}, & \text{ if } a_1 > 1, \\
        [p_0-c_v(1-r)]a_0 p_0^{-a_1}r^{a_2}, & \text{ if } 0 < a_1 \leq 1, 
    \end{cases}
    \label{eq: function G}
\end{equation}
where $c > 0$ is a constant defined in \eqref{eq: constant c}. 
The derivative of $G$ with respect to $r$ is given by
\begin{equation}
    G'(r, a_1) = 
    \begin{cases}
        c(1 - r)^{-a_1} r^{a_2} \left[(a_1 - 1) + (1-r)a_2 r^{-1}\right], & \text{ if } a_1 > 1, \\
        a_0p_0^{-a_1}r^{a_2} \left[c_v + a_2 r^{-1} (p_0 - c_v(1-r))\right], &\text{ if } 0 < a_1 \leq 1. 
    \end{cases}
    \label{eq: function G prime}
\end{equation}

With these extensions, the non-linear ODE \eqref{eq: ode} can be reformulated as
% \begin{equation}
%     \begin{split}
%         \frac{1}{2}\sigma^2 V_2'' +
%         (\gamma - 1)(1 - r)^{\frac{\gamma}{\gamma - 1}} F(V_2') - \delta r V_2' 
%         +G(r, a_1)-\alpha V_2=0.  
%     \end{split}
%     \label{eq: ode with extensions}
% \end{equation}
\begin{equation}
    \begin{split}
        \frac{1}{2}\sigma^2 Q'' +
        (\gamma - 1)(1 - r)^{\frac{\gamma}{\gamma - 1}} F(Q') - \delta r Q' 
        +G(r, a_1)-\alpha Q=0,  
    \end{split}
    \label{eq: ode with extensions}
\end{equation}
which accommodates both regimes of $a_1 > 1$ and $0<a_1 < 1$, and handles the sign of the first-order derivative $Q'$ gracefully. 
Throughout, it suffices to consider the extension \eqref{eq: ode with extensions}, which encapsulates four distinct scenarios based on the values of $a_1$ and the behavior of $Q'$.

%%%%%%%%%%%%%%%%%%%%%%%%%%%%%
\section{Optimility}\label{section4}

In this section, we conduct the stochastic analysis of the control problem \eqref{equwithbm} to establish the connection between the diffusion control formulation and its associated HJB equation. In particular, we demonstrate that the optimal solution to the control problem \eqref{equwithbm} can be obtained by solving the corresponding HJB equation \eqref{equ: formal HJB}.
To this end, we show that the solution to the HJB equation \eqref{equ: formal HJB} coincides with the value function of the stochastic control problem \eqref{equwithbm}. Our approach proceeds in three steps: first, we derive an upper bound for the value function $V_2(r_0)$, as defined in \eqref{equwithbm}, via a verification lemma; second, we prove that this upper bound is attainable by a specific pair of control functions $(u^*, p^*)$; and finally, we verify that these controls are indeed optimal.

First, we exhibit a verification lemma, which indicates that the value function has an upper bound. 

\begin{lemma}[Verification Lemma]
    \label{veriflemma}
    Let $Q:[0, 1]\mapsto \mathbb{R}$ be a bounded twice continuously differentiable solution to \eqref{equ: formal HJB}, which satisfies that $Q'(x) \geq 0$ is bounded for $x\in[0, 1]$. 
    Then $Q$ is an upper bound of the value function $V_2$ defined in \eqref{equwithbm} such that 
    \begin{equation}
        V_2(x) \leq Q(x) \text{ for all } x\in[0, 1]. 
    \end{equation}
\end{lemma}

\begin{proof}
We define a function $f(t, x) = e^{-\alpha t} Q(x)$, where $Q$ satisfies the assumptions mentioned above and $\alpha>0$ is a fixed constant. 
Employing the It\^{o}'s formula to obtain
\begin{equation*}
\begin{aligned}
    e^{-\alpha t}Q(r(t)) &= Q(r_0) + \int_0^t e^{-\alpha s} \left[\frac{\sigma^2}{2} Q''(r(s)) + Q'(r(s)) R(u(s), r(s)) - \alpha Q(r(s))\right]ds \\
    &\quad + \sigma \int_0^t e^{-\alpha s} Q' dW(s) 
    + \int_0^t e^{-\alpha s} C_LdL(s). 
\end{aligned}
\end{equation*}
Observe that the sufficient smoothness of $Q(\cdot)$ and the boundedness of $Q'(\cdot)$ imply that the stochastic integral term has an expected value of zero, since one can obtain 
\begin{equation*}
\begin{aligned}
    E\left[ \int_0^t (\sigma e^{-\alpha s}Q'(r(s)))^2ds \right] &= \sigma^2 E\left[\int_0^t e^{-2\alpha s}(Q'(r(s)))^2ds \right] \\
    &\leq {M}^2 \sigma^2 \int_0^t e^{-2\alpha s} ds < \infty,  
\end{aligned}
\end{equation*}
where $M>0$ serves as the upper bound of $Q'$. 
This further guarantees that the expected value of the stochastic integral is equal to zero. 
Now, taking the expected value on both sides to obtain the following equality:
\begin{equation*}
\begin{aligned}
    E\left[e^{-\alpha t}Q(r(t))\right] &= Q(r_0) + E\bigg[\int_0^{t} e^{-\alpha s} \left(\frac{\sigma^2}{2} Q''(r(s)) + Q'(r(s)) R(u(s), r(s)) - \alpha Q(r(s)) \right)ds \bigg] \\
    &\quad + C_L E\left[\int_0^t e^{-\alpha s}dL(s)\right]. 
\end{aligned}
\end{equation*}
Since $Q$ satisfies the equation \eqref{equ: formal HJB}, we have that for all $u \geq 0$ and $p\geq 0$, 
\begin{equation*}
    \begin{aligned}
        \alpha Q &= \sup_{u,p\geq 0} \left\{\frac{1}{2}\sigma^2 Q''(r(s)) + Q'(r(s)) R(u(s), r(s)) + \pi(p(s), u(s), r(s))\right\} \\
        &\geq \frac{1}{2}\sigma^2 Q''(r(s)) + Q'(r(s)) R(u(s), r(s)) + \pi(p(s), u(s), r(s)), 
    \end{aligned}
\end{equation*}
which further implies the following inequality:
\begin{equation*}
    -\pi(p(s), u(s), r(s)) \geq \frac{1}{2}\sigma^2 Q''(r(s)) + Q'(r(s)) R(u(s), r(s)) -\alpha Q. 
\end{equation*}
Therefore, we have 
\begin{equation*}
    E\left[e^{-\alpha t}Q(r(t))\right] \leq Q(r_0) - E\left[\int_0^t e^{-\alpha s} \pi(p(s), u(s), r(s))ds\right] + C_L E\left[\int_0^t e^{-\alpha s}dL(s)\right]. 
\end{equation*}
Let $t$ tend to infinity. Due to the boundedness of $Q(\cdot)$, and with some arrangements, we obtain
\begin{equation*}
% \begin{aligned}
    J(r_0, u, p) = E\left[\int_0^\infty e^{-\alpha s} \left(\pi(p(s), u(s), r(s))ds - C_L dL(s)\right) \right] 
    % &\leq Q(r_0) - E\left[e^{-\alpha t}Q(r(t))\right] \\
    \leq Q(r_0). 
% \end{aligned}
\end{equation*}
We further take the supremum on both sides over all the admissible controls $u, p\geq 0$ to obtain the upper bound of the value function such that 
\begin{equation}
    V_2(r_0) \leq Q(r_0), 
\end{equation}
for $r_0\in[0, 1]$, since the left-hand side coincides with the infinite-horizon discounted cost functional defined in \eqref{eq: cost functional}. 
This completes the proof. 
\end{proof}

The existence and uniqueness of the solution to the HJB equation \eqref{equ: formal HJB} can be established using standard results from the theory of differential equations (see, for example, \cite{wend1969existence}, \cite{evans2022partial}). 
We summarize this result in Proposition~\ref{prop: solution to HJB} below, with the proof deferred to Appendix~\ref{appendix: Proof of Proposition solution to HJB}. 

\begin{prop}
    There is a bounded twice continuously differentiable function $Q:[0, 1]\mapsto \mathbb{R}$ satisfying:
    \begin{enumerate}[(i)]
        % \item $Q$ is non-negative, strictly decreasing, and convex on $[0, \infty)$, 
        
        \item There is an $M>0$ such that $0\leq Q'(x) \leq M$ for all $x\in[0, 1]$, and
        
        \item $Q$ satisfies the formal HJB equation \eqref{equ: formal HJB} on $[0, 1]$.
    \end{enumerate}
    \label{prop: solution to HJB}
\end{prop}

It remains to show that the upper bound established in the verification lemma (Lemma~\ref{veriflemma}) is attainable by selecting specific control functions $u^*, p^* > 0$, which constitute the optimal solution and achieve the maximum in \eqref{equ: formal HJB}. 
The corresponding optimal state process $r^*$ satisfies the following reflected stochastic differential equation:
\begin{equation}
    dr^*(t) = R(u^*(t), r^*(t))dt + \sigma dW(t) + dL^*(t) - dU^*(t). 
    \label{eq: optimal r*}
\end{equation}
The following result, Theorem~\ref{An optimal control}, establishes that the upper bound obtained in Lemma~\ref{veriflemma} is indeed achievable. 

\begin{theorem}
Let $Q(\cdot)$ be the smooth solution to \eqref{equ: formal HJB}. 
Then the ordered quadruple $(r_0, u^*, p^*, r^*)$ is an optimal admissible control system, where a weak solution $r^*$ satisfying the corresponding \eqref{eq: optimal r*} is an optimal state process. 
\label{An optimal control}
\end{theorem}

\begin{proof}
It is straightforward that \eqref{eq: optimal r*} admits a weak solution for all $t\geq0$ (see \cite{karatzas1991brownian} and \cite{oksendal2013stochastic} for more details). 
With the help of \eqref{equ: formal HJB} and the It\^{o}'s formula, we obtain
\begin{equation*}
\begin{aligned}
    e^{-\alpha t}Q(r^*(t)) &= Q(r_0) + \int_0^t e^{-\alpha s} \left[\frac{\sigma^2}{2} Q''(r^*(s)) + Q'(r^*(s)) R(u^*(s), r^*(s)) - \alpha Q(r^*(s))\right]ds \\
    &\quad + \sigma \int_0^t e^{-\alpha s} Q' dW(s)
    +\int_0^t e^{-\alpha s} C_LdL^*(s), 
\end{aligned}
\end{equation*}
where $\alpha>0$ is a fixed constant and the optimal strategies
\begin{equation}
    u^*=\left((1-r^*) Q' \right)^{\frac{\gamma}{\gamma - 1}}, \quad p^*=\frac{a_1 c_v (1-r^*)}{a_1 - 1}.
    \label{eq: optimal u^* and p^*}
\end{equation}
Since the stochastic integral term has an expected value of zero, by taking the expected value on both sides and some simple algebraic manipulations, we can further obtain the following:
\begin{equation*}
    Q(r_0) - E\left[e^{-\alpha t}Q(r^*(t))\right]  =  E\left[\int_0^t e^{-\alpha s} \left(\pi(p^*(s), u^*(s), r^*(s))ds - C_L dL^*(s)\right)\right]. 
\end{equation*}
The boundedness of $Q(\cdot)$ guarantees that $\lim_{t\to\infty} E[e^{-\alpha t}Q(r^*(t))]=0$. 
Further, let $t$ tend to positive infinity to obtain the upper bound, which completes the proof. 
\end{proof}

The results above establish the consistency between the solution to the HJB equation \eqref{equ: formal HJB} and the solution to the stochastic control problem \eqref{equwithbm}. 
Moreover, the optimal control pair $(u^*, p^*)$ is explicitly characterized in \eqref{eq: optimal u^* and p^*}.

\section{Numerical examples}\label{section5}

% \subsection{Classical numerical methods}

In this section, we present concrete numerical examples of the optimal strategies derived in Section~\ref{section4}. These examples allow us to investigate the influence and validity of the parameter $a_1$ on the optimal strategies, particularly in the context of the extended model described in \eqref{eq: ode with extensions}. We consider two distinct cases: $a_1 > 1$ and $0 < a_1 < 1$.

\begin{example}[$a_1 > 1$]
\label{example1}
    We now consider the simulation of the sample solution illustrated in Figure~\ref{fig: W_k with different k values}. Based on this reproducible numerical example, it is natural to construct the corresponding optimal state process $r^*$ and the associated optimal strategies $(u^*, p^*)$, as defined in \eqref{eq: optimal r*} and \eqref{eq: optimal u^* and p^*}, respectively.
    For this simulation, we fix the parameters as follows: $C_L = 0.5$, $\sigma^2 = 2$, $\gamma = 5$, $a_1 = 1.1$, $a_2 = 5$, $\delta = 0.5$, and $c_v = 0.2$. We consider the time horizon $t \in [0, T]$ with $T = 2$, and set the initial recycling rate to $r_0 = 0.5$. Using a time step of $0.002$, we obtain the sample path of the optimal state process shown in Figure~\ref{fig: optimal r star}, while the corresponding optimal strategies evolve over time as depicted in Figure~\ref{fig: optimal u star and p star}.
    \begin{figure}[h!]
        \centering
        \includegraphics[scale=0.5]{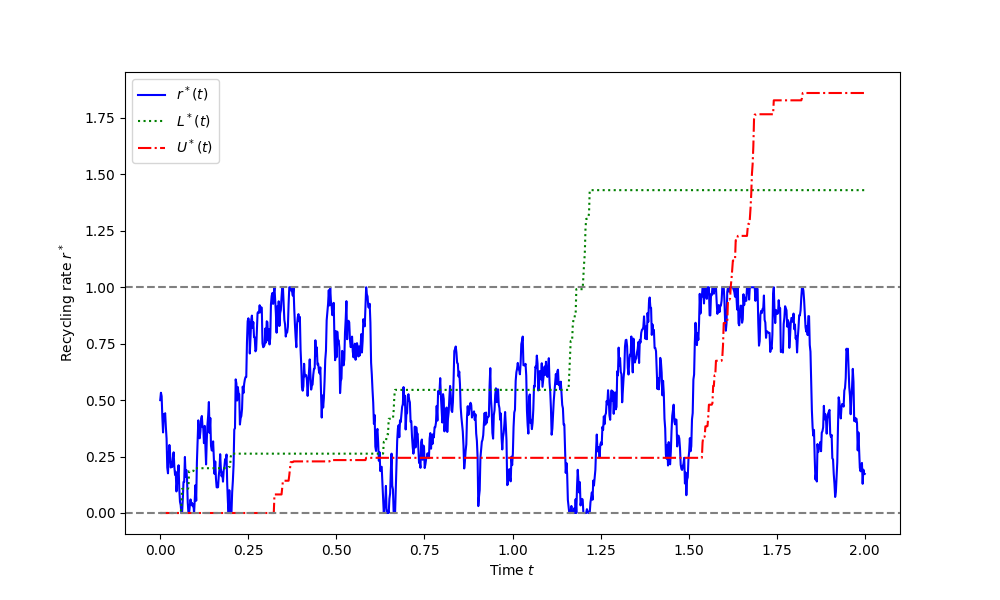}
        \caption{Sample path of optimal state process \eqref{eq: optimal r*}, where the blue solid line exhibits the optimal state process $r^*$, the red dashed line $U^*$ denotes the optimal local time process of $r^*$ reaching the upper boundary 1, and the green dotted line $L^*$ represents the optimal local time process of $r^*$ reaching the lower boundary zero. }
        % this graph has a1 = 1.1
        \label{fig: optimal r star}
    \end{figure}
    \begin{figure}[h!]
        \centering
        \includegraphics[scale=0.5]{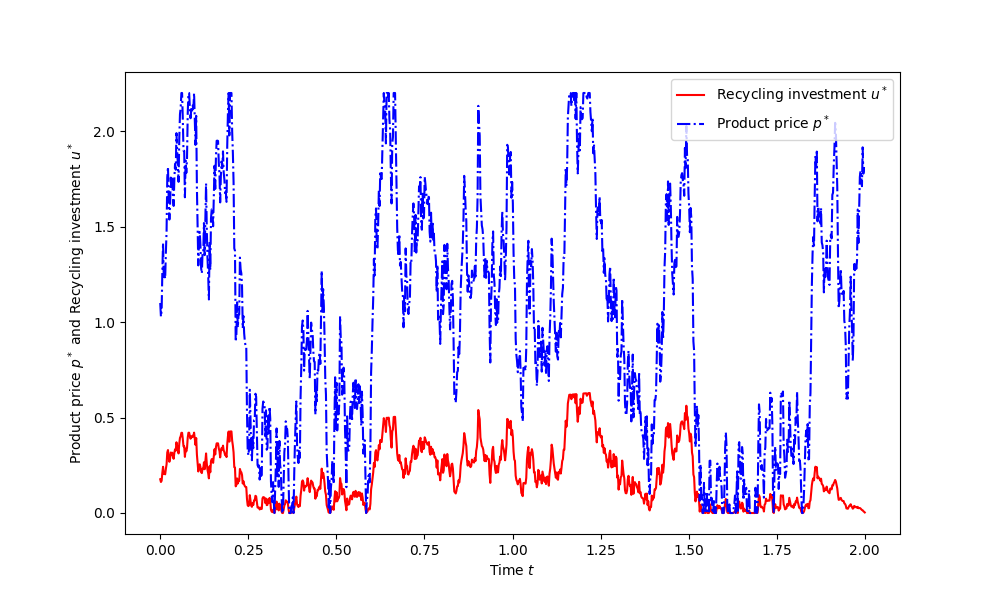}
        \caption{Sample path of optimal policies \eqref{eq: optimal u^* and p^*}, where the blue dashed line exhibits the optimal price $p^*$, and the red solid line $u^*$ denotes the optimal recycling investment. }
        % this graph has a1 = 1.1
        \label{fig: optimal u star and p star}
    \end{figure}
\end{example}

As illustrated in the simulation, the recycling rate remains confined within the interval $[0, 1]$, with the local time processes activated whenever the state process $r^*$ reaches either boundary. Furthermore, the optimal recycling investment $u^*$ is consistently bounded above by the product price $p^*$, and both are influenced by the evolution of the recycling rate $r^*$. According to \eqref{eq: optimal u^* and p^*}, the relationship between $u^*$ and $r^*$ is polynomial of order $\frac{\gamma}{\gamma - 1}$, while $p^*$ depends linearly on $r^*$.

\begin{example}[$0 < a_1 < 1$]

    As a complementary example to the previous case, we now consider the scenario where $a_1 = 0.3$, such that the structure of \eqref{eq: function G} ensures the optimal price $p^*$ attains the upper bound $p_0$ due to its monotonicity (see Remark~\ref{remark: parameter a_1 <= 1}). All other parameters are kept unchanged. 
    This example illustrates a steady optimal pricing policy, as shown in Figure~\ref{fig: optimal u star and p_0 star for small a_1}, where the price remains fixed at $p_0$ and dominates the recycling investment $u^*$ throughout the time horizon.
    \begin{figure}[h!]
        \centering
        \includegraphics[scale=0.5]{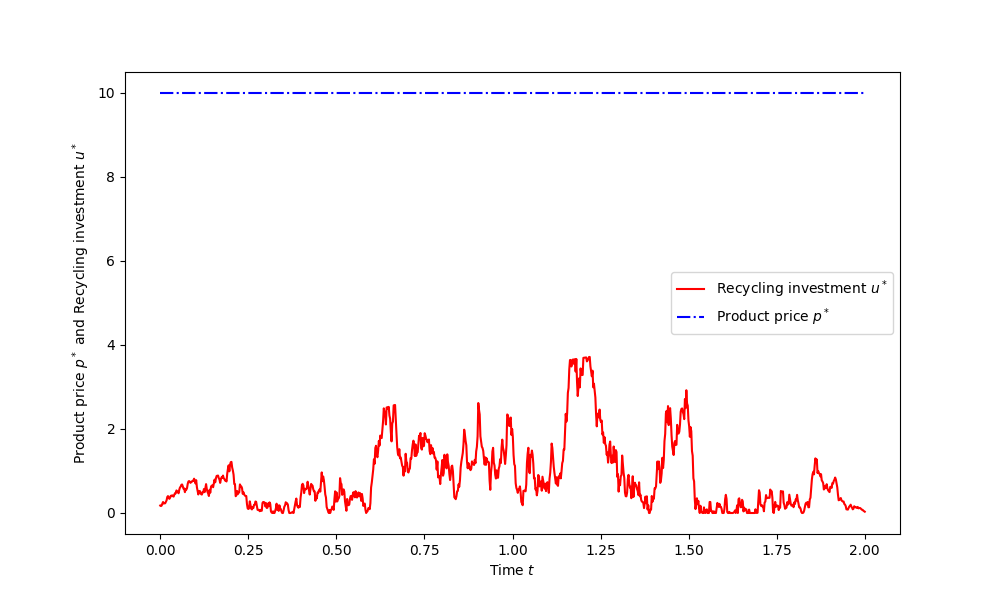}
        \caption{Sample path of optimal policies \eqref{eq: optimal u^* and p^*} when $0 < a_1 \leq 1$, where the blue dashed line exhibits the constant optimal price $p^*$, and the red solid line $u^*$ denotes the optimal recycling investment. }
        % this graph has a1 = 0.3
        \label{fig: optimal u star and p_0 star for small a_1}
    \end{figure}
\end{example}

In Example~\ref{example1}, we presented the optimal solution profile for the state process and its corresponding control strategies. The implementation is primarily based on Proposition~\ref{prop: solution to HJB}, where we introduced a parameterized differential equation associated with the formal HJB equation, along with a free initial condition.
The optimal solution is obtained by identifying an optimal initial condition $k^*$ such that Proposition~\ref{prop: solution to HJB} holds and the cost functional is maximized. The solution profiles for various values of $k$ are illustrated in Figure~\ref{fig: W_k with different k values}.
In the next example, we retain the assumptions from Example~\ref{example1} and explore the influence of varying the initial condition $k$ in comparison to the optimal value $k^*$, as characterized in Proposition~\ref{prop: solution profile for Y_k^*}. For clarity of exposition, we consider three representative values satisfying $k_1 < k^* < k_3$.

\begin{example}[Variation of initial condition in Proposition \ref{prop: solution profile for Y_k^*}]
    We retain the parameters from Example~\ref{example1} and conduct simulations using different initial values, specifically $k_1 = -0.5$ and $k_3 = 0.5$, such that $k_1 < k^* < k_3$. The optimal value $k^*$ is explicitly defined in Equation~\eqref{eq: k^*}. A comparison of the resulting recycling processes is illustrated in Figure~\ref{fig: compare r}. 
    \begin{figure}[h!]
        \centering
        \includegraphics[scale=0.5]{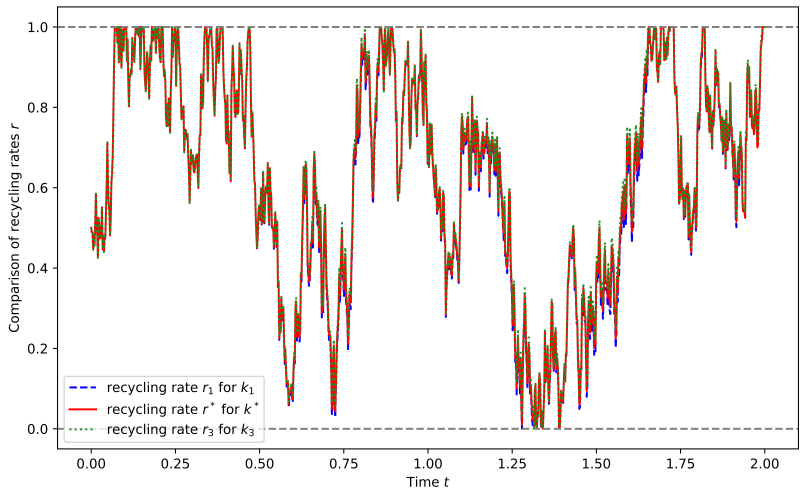}
        \caption{Sample path of state processes with respect to $k_1 < k^* < k_3$ when $0 < a_1 \leq 1$, where the blue dashed line exhibits the recycling process for $k_1$, and the red solid line denotes the optimal recycling process for $k^*$, and the cyan dotted line represents the recycling process for $k_3$. }
        % this graph has a1 = 1.1
        \label{fig: compare r}
    \end{figure}
    Figures~\ref{fig: compare p} and~\ref{fig: compare u} illustrate the comparison of product price $p$ and recycling investment $u$, respectively, under different initial conditions satisfying $k_1 < k^* < k_3$. 

    \begin{figure}[h!]
        \centering
        \includegraphics[scale=0.5]{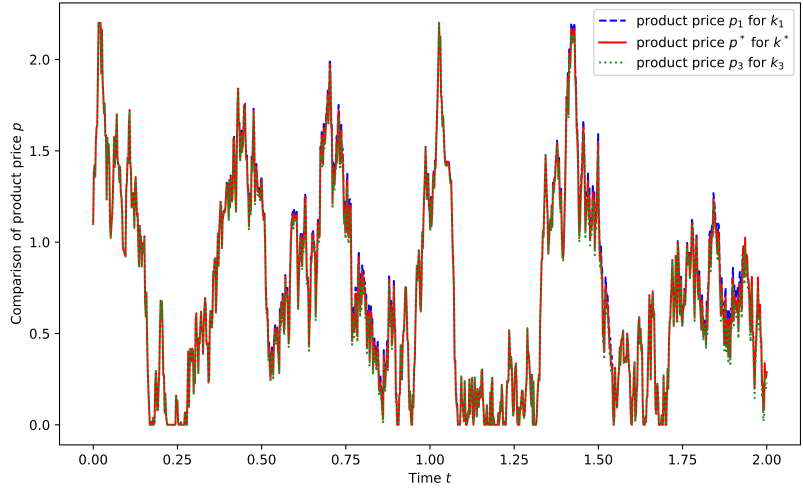}
        \caption{Sample path of product price with respect to $k_1 < k^* < k_3$, where the blue dashed line exhibits product price for $k_1$, the red solid line denotes product price for $k^*$, and the cyan dotted line represents product price for $k_3$. }
        % this graph has a1 = 1.1
        \label{fig: compare p}
    \end{figure}
    \begin{figure}[h!]
        \centering
        \includegraphics[scale=0.5]{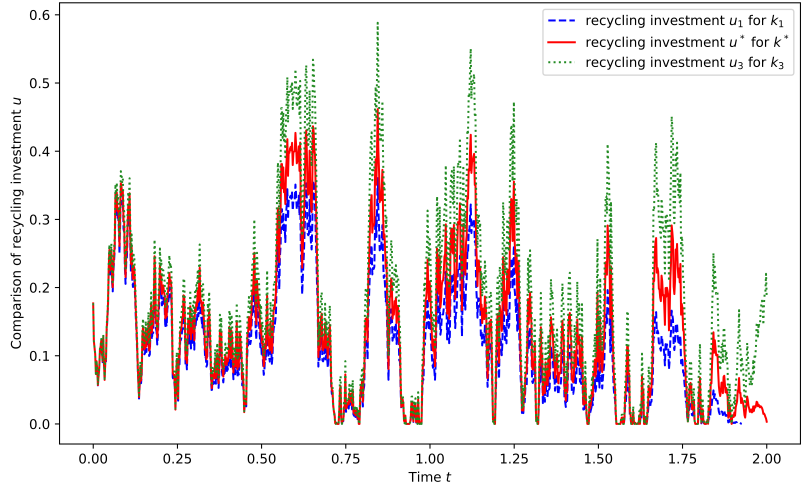}
        \caption{Sample path of recycling investment with respect to $k_1 < k^* < k_3$, where the blue dashed line exhibits recycling investment for $k_1$, the red solid line denotes recycling investment for $k^*$, and the cyan dotted line represents recycling investment for $k_3$. }
        % this graph has a1 = 1.1
        \label{fig: compare u}
    \end{figure}
\end{example}

We observe that the optimal state process and control strategies fluctuate within the bounds defined by the state processes and strategies corresponding to $k_1$ and $k_3$. Figure~\ref{fig: W_k with different k values} presents a comparison of the value functions, generated based on the derivative of $Q$ as defined in Proposition~\ref{prop: solution to HJB}. The simulation for $Q$ itself is omitted, as its variation is negligible compared to that of its derivative. 
Moreover, the figure clearly illustrates the solution profile, aligning with the analytical results discussed in Appendix~\ref{appendix: Proof of Proposition solution to HJB}.

%%%%%%%%%%%%%%%%%%%%%%%%%%%%%%%%%%%%%%%%%%%%%%%%%%%%%%%%%%%%%%%%%%%%%%%%%%%%%%%
% \subsection{Reinforcement Learning in joint stochastic control}\label{sec6 RL}

% \todo[color=orange!30]{*** include RL}

%%%%%%%%%%%%%%%%%%%%%%%%%%%%%%%%%%%%%%%%%%%%%%%%%%%%%%%%%%%%%%%%%%%%%%%%%%%%%%%

\section{Conclusion}\label{section7}

In this paper, we extend the framework of production and process innovation to a general setting influenced by the propagation of external Gaussian chaos. To ensure the practicality of the stochastic model, we introduce two key constraints via local time processes, which serve as incentives aligned with green economy policies.
To investigate the joint dynamics of recycling investment and product price within a regulated stochastic recycling rate model, we formulate an HJB equation and propose suitable function extensions. A detailed derivation is provided to establish the connection between the stochastic control problem and the corresponding HJB equation. To solve the HJB equation, we reduce it to a non-linear ODE and conduct a theoretical analysis to characterize the solution profile and associated control variables. 
Furthermore, we carefully examine the assumptions imposed on system parameters, highlighting their critical role in determining optimal strategies. We also interpret these assumptions in practical terms to ensure their applicability and tractability. 
Finally, we present illustrative numerical analyses to demonstrate the significance of parameter selection in shaping system behavior and outcomes.

There are several potential extensions to the regulated stochastic control problem proposed in this study. 
While we have focused on an autonomous differential equation, time may play a critical role in practical applications, potentially influencing system dynamics. In this context, a time-dependent formulation could be considered, leading to a partial differential equation (PDE). Although the analysis of such systems may be more complex, the distinctions and advantages between autonomous and time-dependent models merit further exploration. 
Additionally, the efficiency of algorithms for solving joint stochastic control problems remains an open question. In our future work, we aim to evaluate the performance and stability of non-conventional approaches. For example, reinforcement learning (RL) algorithms by comparing them with theoretical results, thereby demonstrating the practical applicability of the regulated stochastic model in real-world scenarios. 
These works will be considered in our subsequent research. 
Relevant studies employing deep learning techniques to solve PDEs and stochastic control problems include, for example,  \cite{baydin2018automatic}, \cite{purwins2019deep}, \cite{raissi2019physics}, \cite{zhang2019quantifying}, \cite{mao2020physics}, \cite{meng2020ppinn}, \cite{JagtapAD2020XPINNs}, \cite{raissi2020hidden}, \cite{lu2021deepxde}, \cite{dai2022queueing}, \cite{xie2025single}.

%\section{Acknowledgement}
%This work  was partially supported by

% %%%%%%%%%%%%%%%%%%%%%%%%%%%%%%
\newpage
\appendix

\section{Proof of Proposition \ref{prop: solution to HJB}}
\label{appendix: Proof of Proposition solution to HJB}

It suffices to consider \eqref{eq: ode with extensions}, as any solution to this equation can be shown to also satisfy the formal HJB equation~\eqref{equ: formal HJB}. Notably, \eqref{eq: ode with extensions} is a non-linear second-order ordinary differential equation, which can be interpreted as a boundary value type problem.
To analyze the solution profile, we impose an initial condition and parameterize the equation accordingly, enabling a more tractable investigation of its behavior. 

To this end, we consider a family of solutions $\{Y_k\}_{k \in \mathbb{R}}$ to the following differential equation parameterized by $k \in \mathbb{R}$: 
\begin{equation}
\begin{aligned}
    \frac{\sigma^2}{2} Y_k'' + (\gamma - 1)(1 - x)^{\frac{\gamma}{\gamma - 1}} F(Y_k') 
    -\delta x Y_k' 
    + G(x, a_1) &= \alpha Y_k, \\
    Y_k'(0) &= C_L, \\
    Y_k'(1) &= 0.  
    \label{eq: parameterized Y_k}
\end{aligned}
\end{equation}
Our objective is to find a suitable $k \in \mathbb{R}$ so that 
\begin{enumerate}[(i)]
    \item There exists a solution $Y_k(\cdot)$ to \eqref{eq: parameterized Y_k} defined on $[0, 1]$, and
    
    \item $Y_k(x)$ and $Y_k'(x)$ are bounded for all $x\in [0, 1]$. 
\end{enumerate}
We intend to reveal the solution profile and its properties in terms of the parameter $k$ along with the variation of $x$ values. 

For the ease of analysis, we introduce $W_k:= Y_k'$, which further renders
\begin{equation}
    \frac{\sigma^2}{2} W_k' + (\gamma - 1) (1 - x)^{\frac{\gamma}{\gamma - 1}} F(W_k) - \delta x W_k + G(x, a_1) = \alpha \int_0^x W_k(s)ds + \alpha K_k, 
    \label{eq: W_k with integral}
\end{equation}
where $K_k$ is a constant obtained by integrating the above substitution such that $Y_k(x) = \int_0^x W_k(s)ds + K_k$ for each $x \in [0, 1]$. 
Moreover, we can derive a representation of $K_k = \frac{1}{\alpha}(\frac{\sigma^2}{2}W_k'(0) + (\gamma - 1) F(C_L))$ for later analysis. 
Its initial conditions are given by $W_k(0) = C_L$, and in addition, we impose $W'_k(0) = k$. 
% , where $W'_k(0) = k$ and $W_k(0) = C_L$. 
Here, we employ the notation $K_k$ to exhibit the linear dependence of the constant $K$ on the parameter $k$. 

To exhibit the existence, we can rewrite the above equation as a system such that 
\begin{equation}
\begin{cases}
    Y_{k}' &= W_k, \\
    W_k' &= \frac{2}{\sigma^2} \left[(1 - \gamma) (1 - x)^{\frac{\gamma}{\gamma - 1}} F(W_k) + \delta x W_k - G(x, a_1) + \alpha Y_k\right], 
    \end{cases}
\end{equation}
with initial conditions $W_k(0) = Y_k'(0) = C_L$, $W_k'(0) = k$, and $Y_k(0) = K_k$. 
It is straightforward that there exists a unique solution to initial-value problems for fixed $k$ (cf. \cite{coddington1955theory}, \cite{brauer1989qualitative}, \cite{sacks2017techniques}). 
In order to fulfill the other boundary condition, it suffices to find a suitable $k$ such that the above objectives (i) and (ii) are satisfied in addition to $W_k(1)  = 0$.

The following lemma exhibits a comparison result with respect to various values of $k$, and we denote this result as the comparison lemma. 

\begin{lemma}[Comparison lemma]
Let $k_2<k_1$. Consider the solution to \eqref{eq: W_k with integral}, $W_{k_1}(\cdot)$ and $W_{k_2}(\cdot)$, respectively. 
Then, we have $W_{k_1}(x) \geq W_{k_2}(x)$ for all $x\in [0, 1]$ (equality holds only if $x = 0, 1$). 
Moreover, the corresponding $Y_{k_1}(x)$ is greater than $Y_{k_2}(x)$. 
\label{lemma: comparison lemma}
\end{lemma}

\begin{proof}
    Since $W_{k_1}(0) = W_{k_2}(0) = C_L$ and $W_{k_1}'(0) = k_1 > k_2 = W_{k_2}'(0)$, we may introduce a comparison function $f(x) = W_{k_1}(x) - W_{k_2}(x)$ such that $f(0) = 0$ and $f'(0) > 0$. 
    Note that since $f(1) = W_{k_1}(1) -W_{k_2}(1) = 0$, there exists some $\hat{x}\in [0, 1]$ so that $f'(\hat{x}) = 0$. 
    Suppose that there is an $a_0\in (0, 1)$ so that $f(a_0) < 0$. 
    This implies that the paths of $f(\cdot)$ across the $x$-axis, i.e., we can find an $a \in (0, a_0)$ so that $f$ crosses $x$-axis for the first time so that $f(a) = 0$ and $f'(a) \leq 0$. 
    If we consider the difference of \eqref{eq: W_k with integral} with respect to $k_1$ and $k_2$ at $x=a$, we have
    \begin{align*}
    &\quad \frac{\sigma^2}{2} (W_{k_1}'(a) - W_{k_2}'(a)) 
    + (\gamma - 1) (1 - a)^{\frac{\gamma}{\gamma - 1}} \left(F(W_{k_1}(a)) - F(W_{k_2}(a))\right)
    -\delta a (W_{k_1}(a) - W_{k_2}(a)) \\
    &= \alpha \int_0^a (W_{k_1}(s) - W_{k_2}(s))ds + \alpha (K_{k_1} - K_{k_2}), 
    \end{align*}
    which further implies 
    \begin{equation}
        \frac{\sigma^2}{2} f'(a) 
        + (\gamma - 1) (1 - a)^{\frac{\gamma}{\gamma - 1}} \left(F(W_{k_1}(a)) - F(W_{k_2}(a))\right)
        -\delta a f(a) = \alpha \int_0^a f(s)ds + \alpha (K_{k_1} - K_{k_2}). 
        \label{eq: difference of W_{k_1} and W_{k_2}}
    \end{equation}
    Notice that the left-hand side is non-positive since $f'(a)\leq0,f(a) = 0$, and $F(W_{k_1}(a)) - F(W_{k_2}(a)) = 0$. However, the right-hand side indicates a strictly positive value since $K_{k_1} - K_{k_2} = \frac{\sigma^2}{2\alpha} (k_1 - k_2) > 0$. The contradiction suggests that such $a_0$'s do not exist. Therefore, $f(\cdot)$ does not take negative values. 

    Further, we intend to rule out the case that a local minimum presents on the $x$-axis tangentially. Suppose there is an $a\in [0, 1]$ so that $f(a) = 0$ and $f'(a) = 0$. 
    We observe that the left-hand side of \eqref{eq: difference of W_{k_1} and W_{k_2}} vanishes since $f'(a) = 0$, $F(W_{k_1}(a)) - F(W_{k_2}(a)) = 0$, and $f(a) = 0$. However, this contradicts the strictly positive right-hand side. 
    Thus, $f(x) \geq 0$ for all $x\in[0, 1]$. That is $W_{k_1}(x) \geq W_{k_2}(x)$ and equality holds only when $x = 0, 1$. 
    Since $Y_k(x) = \int_0^x W_k(s)ds + K_k$, we have $Y_{k_1}(x) - Y_{k_2}(x) > 0$. This completes the proof. 
\end{proof}

Next, we concern the profile of $W_k(\cdot)$ for fixed $k \in \mathbb{R}$ and its properties. 
To this end, we rewrite \eqref{eq: W_k with integral} as a second-order equation: 
\begin{equation}
    \frac{\sigma^2}{2} W_k'' - \gamma (1 - x)^{\frac{1}{\gamma - 1}} F(W_k) 
    + (\gamma - 1) (1 - x)^{\frac{\gamma}{\gamma - 1}} F'(W_k) W_k'  - \delta x W_k' 
    +G'(x, a_1)
    = (\alpha + \delta) W_k, 
\label{eq: W_k second order}
\end{equation}
where $F'$ and $G'$ are defined in \eqref{eq: function F prime} and \eqref{eq: function G prime}, respectively. 
Fix $k \in \mathbb{R}$, 
notice that if $W_k'(\xi) = 0$ for some $\xi \in [0, 1]$, \eqref{eq: W_k second order} at $\xi$ implies
\begin{equation}
    \frac{\sigma^2}{2} W_k''(\xi) - \gamma (1 - \xi)^{\frac{1}{\gamma - 1}} F(W_k(\xi))
    + G'(\xi, a_1)
    = (\alpha + \delta) W_k(\xi). 
    \label{eq: ode with w'(xi) = 0}
\end{equation}
In addition, if $W_k(\xi) = 0$ for some $\xi \in(0, 1]$, we have 
\begin{equation}
    \frac{\sigma^2}{2} W_k''(\xi) 
    =- G'(\xi, a_1)
    \label{eq: W_k second when W_k(xi) = 0}
\end{equation}
The convexity of $W_k$ at $\xi$ depends on the sign of $G'(\xi, a_1)$ and the constrains of $a_1$ and $a_2$, which are in twofold: $a_1 >1$ and $0 < a_1 \leq 1$, by \eqref{eq: function G prime}. 

If $a_1 > 1$, then the convexity relies on the sign of $(a_1 - 1) + (1-\xi) a_2 \xi^{-1}$. Note that since $\xi\in(0, 1)$ and $a_2 >0$, it is guaranteed that $(a_1 - 1) + (1-\xi) a_2 \xi^{-1} > 0$. More precisely, the sign depends on the restriction of $a_1$ and $a_2$, which are exhibited as follows: 
first, if $a_2 > a_1 - 1$, then we have $\frac{a_2}{a_2 - a_1 + 1} > 1$. Thus, $\xi < 1 < \frac{a_2}{a_2 - a_1 + 1}$, which yields $(a_1 - 1)+(1 - \xi) a_2 \xi^{-1} > 0$. By \eqref{eq: function G prime} and \eqref{eq: W_k second when W_k(xi) = 0}, we have $W_k''(\xi) < 0$, i.e., there exists a zero local maximum;  
% As a consequence, $W_k$ also has a negative local minimum. 
second, if $a_2 < a_1 - 1$, then we obtain $\frac{a_2}{a_2 - a_1 + 1} < 0$. 
Observe that $\frac{a_2}{a_2 - a_1 + 1} < 0 < \xi$, which further implies $(a_1 - 1) + (1-\xi) a_2 \xi^{-1} > 0$. 
If $a_2 = a_1 - 1$, then $(a_1 - 1)+(1 - \xi) a_2 \xi^{-1} = \frac{a_2}{\xi} > 0$. 
Therefore, we have $W_k''(\xi) < 0$, and a zero local maximum exists.

If $0 < a_1 \leq 1$, then by \eqref{eq: function G prime}, the convexity depends on the sign of $c_v + a_2 \xi^{-1}[p_0 - c_v(1-\xi)]$. 
Note that since it is assumed in Remark \ref{remark: parameter a_1 <= 1} that the unit production cost of virgin resources does not exceed the greatest retail price, namely, $p_0 \geq c_v$, it ensures a strictly positive $c_v + a_2 \xi^{-1}[p_0 - c_v(1-\xi)]$.
More precisely, we have $\frac{a_2(c_v - p_0)}{(a_2 + 1)c_v} \leq 0 < \xi$, which yields $c_v + a_2 \xi^{-1}[p_0 - c_v(1-\xi)] > 0$. Thus, $W_k''(\xi) < 0$, and there is a zero local maximum.

\begin{remark}
\label{remark: p_0 < c_v}
    As a complement to the argument in the case of $0 < a_1 \leq 1$, if we assume $p_0 < c_v$, then we obtain $0 < \frac{a_2(c_v - p_0)}{(a_2 + 1)c_v} < 1$. 
    In this case, the system manager ignores the fact that the profit per unit might be negative if the greatest retail price $p_0$ is quite low when taking dumping as a predatory pricing strategy. 
    We have the following cases:  
    First, if $\xi > \frac{a_2(c_v - p_0)}{(a_2 + 1)c_v}$, then we have $c_v + a_2 \xi^{-1}[p_0 - c_v(1-\xi)] > 0$, which further implies $W_k''(\xi) < 0$. Hence, there is a zero local maximum. 
    Second, if $\xi < \frac{a_2(c_v - p_0)}{(a_2 + 1)c_v}$, then we have $c_v + a_2 \xi^{-1}[p_0 - c_v(1-\xi)] < 0$, which indicates $W_k''(\xi) > 0$. Hence, there is zero local minimum. 
    Last but not least, if $\xi = \frac{a_2(c_v - p_0)}{(a_2 + 1)c_v}$, i.e., $c_v + a_2 \xi^{-1}[p_0 - c_v(1-\xi)] = 0$, it leads to an identically zero solution by uniqueness, which contradicts the non-zero initial condition. 
    For ease of exposition and to guarantee the practicality of our assumptions along with the general concrete environment, we omit the situation of other abnormal predatory pricing strategies for the rest of the discussion. 
\end{remark}

% -----------------------
% \textcolor{red}{To be deleted}
% If $0 < a_1 \leq 1$, then the convexity depends on the sign of $c_v + a_2 \xi^{-1}[p_0 - c_v(1-\xi)]$ by \eqref{eq: function G prime}. Since $\frac{a_2(c_v - p_0)}{(a_2 + 1)c_v} < 1$, a similar Case 1 does not occur. 
% We intend to preserve a similar discussion as in the previous Case 2 as follows: 
% first, if $\xi > \frac{a_2(c_v - p_0)}{(a_2 + 1)c_v}$, then we have $c_v + a_2 \xi^{-1}[p_0 - c_v(1-\xi)] > 0$, which further implies $W_k''(\xi) < 0$. Hence, there is a zero local maximum; 
% second, if $\xi < \frac{a_2(c_v - p_0)}{(a_2 + 1)c_v}$, then we have $c_v + a_2 \xi^{-1}[p_0 - c_v(1-\xi)] < 0$, which indicates $W_k''(\xi) > 0$. Hence, there is zero local minimum; 
% last but not least, if $\xi = \frac{a_2(c_v - p_0)}{(a_2 + 1)c_v}$, i.e., $c_v + a_2 \xi^{-1}[p_0 - c_v(1-\xi)] = 0$, it leads to a contradiction as in Case 2 due to uniqueness and non-zero initial condition. 

Observe that in the situations of $W_k(\xi)  > 0$ and $W_k(\xi) < 0$, it is straightforward to deduce similar results by following the same fashion. For brevity, we omit the detailed discussions.

These cases provide a general overview of the solution profile. We observe that when $a_1 > 1$, the solution does not admit any zero or negative local minima. A similar observation holds when $0 < a_1 \leq 1$, indicating the absence of negative oscillations. 
In the following lemma, we demonstrate that if the solution $W_k$ to \eqref{eq: W_k with integral}, associated with some $k \in \mathbb{R}$, intersects the $x$-axis, it cannot do so tangentially. In other words, $W_k$ and its derivative $W_k'$ cannot vanish simultaneously.

% Note that these cases provide a general picture of the solution profile. 
% % \textcolor{red}{
% We observe that when $a_1 > 1$, there cannot be any zero local minimum nor negative local minimum. When $0<a_1\leq 1$,  
% % and $p_0 \geq c_v$, 
% we have an identical observation. These imply that there is no negative oscillation. 
% % Moreover, when $0<a_1 \leq 1$ and $p_0 < c_v$, we observe that if a critical point appears to be greater than $\frac{a_2(c_v - p_0)}{(a_2 + 1)c_v}$, there cannot be any zero local minimum nor negative local minimum. If a critical point happens to be less than $\frac{a_2(c_v - p_0)}{(a_2 + 1)c_v}$, then there cannot be any zero local maximum nor positive local maximum. These observations indicate that there is no oscillation for the negative and positive parts of the solution accordingly, which may depend on the value of the critical point and $\frac{a_2(c_v - p_0)}{(a_2 + 1)c_v}$.
% % } 
% % In a nutshell, the positive critical points depend on the relationship between $a_1$ and $a_2$ as well as $\frac{a_2(c_v - p_0)}{(a_2 + 1)c_v}$. 
% In the next lemma, we demonstrate that whenever the solution $W_k$ to \eqref{eq: W_k with integral} associated with some $k\in \mathbb{R}$ achieves $x$ axis, it cannot reach $x$ axis tangentially. Loosely speaking, $W_k$ and $W_k'$ cannot vanish simultaneously. 

\begin{lemma}
    Assume  $W_k$ is the solution for the equation \eqref{eq: W_k with integral} with initial conditions $W_k(0)=C_L$ and $W_k'(0)=k\in\mathbb{R}$. 
    We define the notation $c_k:= \sup\{0 < x < 1: W_k(u) > 0 \text{ for all } 0 \leq u\leq x\}$ as the first time $W_k$ crosses $x$ axis. 
    If $\lim_{x\to c_k^-}W_k(x)=0$ for some $k$, we have $\lim_{x\to c_k^-}W_k'(x) < 0$. 
    \label{lemma: touch x axis tangentially}
\end{lemma}

\begin{proof}
    % We exhibit our proof in two cases: $a_1 > 1$ and $0 < a_1 \leq 1$, due to the definition of \eqref{eq: function G}. 
    % First, we consider the situation of $a_1 > 1$. 
    We assume $\lim_{x\to c_k^-}W_r(x)=0$ for some $k$. Note that \eqref{eq: W_k with integral} guarantees $\lim_{x\to c_k^-}W_k'$ exists and 
    \begin{equation*}
    \begin{aligned}
        \lim_{x\to c_k^-}\frac{\sigma^2}{2}W_k'(x) 
        &= \alpha \int_0^{c_k} W_k(s)ds + \alpha K_k 
        % - (\gamma - 1) \lim_{x\to c_k^-}(1 - x)^{\frac{\gamma}{\gamma - 1}} (W_k)^{\frac{\gamma}{\gamma -1}} 
        % + \delta c_k \lim_{x\to c_k^-}W_k 
        - G(c_k, a_1). 
    \end{aligned}
    \end{equation*}
    Suppose $\lim_{x\to c_k^-}W_k'(x) = 0$, which indicates that $W_k$ reaches $x$ axis tangentially at point $x=c_k$. 

    Here, we consider two separate cases: $a_1 > 1$ and $0< a_1 \leq 1$, due to the definition of \eqref{eq: function G}. 
    If $a_1 > 1$, \eqref{eq: function G} suggests $G(x, a_1) = c(1 - x)^{1-a_1}x^{a_2}$.  
    When $a_2 < a_1 - 1$, we have $\frac{a_2}{a_2 - a_1 + 1} < 0 < c_k$, which suggests $(a_1 - 1) + (1 - c_k) a_2 c_k^{-1} > 0$. 
    % then we have $(a_1 - 1) + (1 - c_k) a_2 c_k^{-1} < 0$. This and in conjunction with \eqref{eq: function G} further imply that 
    % \begin{equation*}
    %     \begin{aligned}
    %         &\quad \lim_{\delta_0 \to 0} \frac{c(1 - c_k)^{1-a_1} c_k^{a_2} - c(1 - c_k + \delta_0)^{1-a_1} (c_k - \delta_0)^{a_2}}{W_k(c_k - \delta_0)} 
    %         \\
    %         &= \lim_{\delta_0 \to 0} \frac{-c_k^{a_2} (1 - c_k)^{-a_1} c \left((a_1 - 1) + (1 - c_k) a_2 c_k^{-1}\right)}{W_k'(c_k - \delta_0)}
    %         \to \infty, 
    %     \end{aligned}
    % \end{equation*}
    % as $\delta_0 \to 0$. Therefore, $\lim_{\delta_0 \to 0} \frac{W_r'(c_k-\delta_0)}{W_k(c_k-\delta_0)} = \infty$, i.e., for any $\Tilde{M} > 0$, there exists a $\Tilde{\delta}$ so that $\frac{W_r'(c_k-\delta_0)}{W_k(c_k-\delta_0)} > \Tilde{M}$ whenever $|\delta_0| < \Tilde{\delta}$. 
    % This implies that $W_k'(c_k - \delta_0) > \Tilde{M} W_k(c_k - \delta_0)$. 
    % Since $W_k'(c_k - \delta_0) < 0$ and $\Tilde{M} > 0$, we have $W_k(c_k - \delta_0) < 0$ which contradicts the definition of $c_k$. 
    % Therefore, $\lim_{x\to c_k^-}W_k'(x) < 0$ holds. 
    % Second, if $c_k < \frac{a_2}{a_2 - a_1 + 1}$, then we have $(a_1 - 1)+(1 - c_k) a_2 c_k^{-1} > 0$. 
    % Instead of the quotient, 
    We consider \eqref{eq: W_k second order}. 
    Taking the limit as $x\to c_k^-$ and by \eqref{eq: function G prime}, we have
    \[
    \lim_{x\to c_k^-} \frac{\sigma^2}{2} W_k''(x) = 
    -c(1 - c_k)^{-a_1} c_k^{a_2} \left[(a_1 - 1) + (1 - c_k) a_2 c_k^{-1}\right] < 0, 
    \]
    which implies that $W_k$ is concave at $c_k$, i.e., $W_k$ has a zero local maximum at $c_k$. This together with $\lim_{x\to c_k^-} W_k(x) = 0$ and $\lim_{x\to c_k^-} W_k'(x) = 0$, we observe that $W_k(c_k - \delta_0) < 0$ for some small $\delta_0$, which contradicts the definition of $c_k$. 
    % In addition, if $c_k = \frac{a_2}{a_2 - a_1 + 1}$, then we have $(a_1 - 1)+(1 - c_k) a_2 c_k^{-1} = 0$. By \eqref{eq: W_k second order}, we have $\lim_{x\to c_k^-} \frac{\sigma^2}{2} W_k''(x) = 0$, which further renders a contradiction by uniqueness and the initial condition since we cannot have an identically zero solution.  
    % The previous Case 1 renders a contradiction. 
    Hence, we have $\lim_{x\to c_k^-} W_k'(x) < 0$. 
    
    It is worth mentioning that all the above cases are discussed under the situation that $a_2 < a_1 - 1$ so that $\frac{a_2}{a_2 - a_1 + 1} < 0$. 
    In the case of $a_2 > a_1 - 1$, i.e., $\frac{a_2}{a_2 - a_1 + 1} > 1$, we have $c_k < 1 < \frac{a_2}{a_2 - a_1 + 1}$. 
    It is straightforward to deduce a contradiction by following a similar fashion. 
    Moreover, a contradiction occurs in the case of $a_2 = a_1 -1$ as well. 
    % , we have $(a_1 - 1)+(1 - c_k) a_2 c_k^{-1} = \frac{a_2}{c_k} >0$. 

    If $0 < a_1 \leq 1$, \eqref{eq: function G} implies $G(x, a_1) = [p_0-(1-x)c_v]a_0 p_0^{-a_1} x^{a_2}$. Taking the limit as $x\to c_k^-$ on both sides of \eqref{eq: W_k second order} to obtain
    \[
        \lim_{x\to c_k^-} \frac{\sigma^2}{2} W_k'' = - G'(c_k, a_1) = - a_0p_0^{-a_1}c_k^{a_2} \left[c_v + a_2 c_k^{-1} (p_0 - c_v(1-c_k))\right]. 
    \]
    Since we assumed $p_0 \geq c_v$, we have $\frac{a_1(c_v - p_0)}{(a_1 + 1)c_v} \leq 0 < c_k$, which in turn implies $c_v + a_2 c_k^{-1} (p_0 - c_v(1-c_k)) > 0$.  Therefore, $\lim_{x\to c_k^-} W_k'' < 0$ so that $W_k$ has a zero local maximum at $c_k$ and $W_k(c_k - \delta_0) < 0$. 
    % As in the previous discussion, we observe that $\frac{a_1(c_v - p_0)}{(a_1 + 1)c_v} < 1$. Hence, it suffices to compare $c_k$ and $\frac{a_1(c_v - p_0)}{(a_1 + 1)c_v}$, namely, the sign of $c_v + a_2 r^{-1} (p_0 - c_v(1-r))$, which in turn leads to the sign of second derivative at $c_k^-$. More precisely, if $c_k > \frac{a_2(c_v - p_0)}{(a_2 + 1)c_v}$, we have $c_v + a_2 r^{-1} (p_0 - c_v(1-r)) > 0$, which further implies $\lim_{x\to c_k^-}W_k'' < 0$. That is, $W_k$ has a zero local maximum at $c_k$ and $W_k(c_k - \delta_0) < 0$. 
    However, this contradicts the definition of $c_k$. 
    As a consequence, $\lim_{x\to c_k^-} W_k' < 0$. 
    This completes the proof. 
\end{proof}

We have characterized the solution profile of $W_k$ with respect to different values of the parameter $k$ and the independent variable $x$ in Lemma~\ref{lemma: comparison lemma} and Lemma~\ref{lemma: touch x axis tangentially}. Next, we aim to investigate the behavior of $W_k$ as $k$ becomes arbitrarily small or large. These results will further aid in visualizing the trajectories of the solutions $W_k$ for varying values of $k$.

\begin{lemma}
    Consider the sequence of solutions $\{W_k\}_{k \in \mathbb{R}}$ to \eqref{eq: W_k second order}. 
    There exists $k_1 < 0$ small enough such that for each $k < k_1$, we have $\lim_{x\to c_k^-}W_k(x) = 0$. 
    \label{lemma: small k}
\end{lemma}

\begin{proof}
    % When $k=0$, \eqref{eq: W_k second order} at $r=0$ suggests that 
    % \[
    % \frac{\sigma^2}{2} W_0''(0) = (\alpha+ \delta) C_L + \gamma C_L^{\frac{\gamma}{\gamma - 1}} > 0,
    % \]
    % which implies that $W_0(\cdot)$ has a local minimum at the origin. 
    We intend to show that it is possible for $W_k$ to cross the $x$ axis for arbitrarily small $k$. To this end, we assume $W_k(x) > 0$ for all $x\in [0, 1]$ and small $k$.  When $x$ is in a small neighborhood of the origin such that $x\in [0, \delta_0)$ where $\delta_0 > 0$, we have $W_k'(0) = k < 0$ and $W_k(x) < W_k(0) = C_L$ since $W_k$ is strictly decreasing over the neighborhood of the origin. 
    Further, if $a_1 > 1$, \eqref{eq: W_k with integral} deduces
    \begin{equation*}
        \begin{aligned}
            \frac{\sigma^2}{2} W_k'(x) &= 
            -(\gamma - 1)(1 - x)^{\frac{\gamma}{\gamma - 1}} F(W_k(x)) 
            -G(x, a_1)
            + \delta x W_k(x) + \alpha \int_0^x W_k(u) du + \alpha K_k
            \\
            % &= 
            % -(\gamma - 1)(1 - x)^{\frac{\gamma}{\gamma - 1}} F(W_k(x)) 
            % - c(1 - x)^{1-a_1} x^{a_2}
            % + \delta x W_k(x) + \alpha \int_0^x W_k(u) du + \alpha K_k
            % \\
            &\leq \delta C_L x  + \alpha C_L x + \alpha K_k - G(x, a_1) \mathbbm{1}_{[0 < a_1 \leq 1]},  
        \end{aligned}
    \end{equation*}
    since the non-negativity of $W_k$ and $x\in[0, \delta_0)$. 
    Note that the last term depends on the restriction of $a_1$. If $a_1 > 1$, by \eqref{eq: function G}, we know $G(x, a_1) \geq 0$, and we may eliminate this term as an upper bound. 
    If $0 < a_1 \leq 1$, the last term may be preserved, and it depends on its sign. 
    More precisely, it depends on the following cases: on the one hand, if $x< 1 - \frac{p_0}{c_v}$, then we have $p_0 - (1 - x)c_v < 0$. This together with \eqref{eq: function G} suggests $G(x, a_1) < 0$; On the other hand, if $x \geq 1 - \frac{p_0}{c_v}$, then we have $p_0 - (1 - x)c_v \geq 0$, which implies $G(x, a_1) \geq 0$.  
    For brevity, we postpone the detailed discussion until the end of this proof. 

    We intend to tackle these cases separately. 
    If $a_1 > 1$, we discard the last term of the above inequality. 
    Integrating both sides over $[0, s]$, we obtain
    \begin{equation*}
            W_k(s) \leq C_L + \frac{2}{\sigma^2} \left( \delta C_L \frac{s^2}{2} + \alpha C_L \frac{s^2}{2} + \alpha K_k s \right), 
    \end{equation*}
    where $K_k = \frac{1}{\alpha} (\frac{\sigma^2}{2} k + (\gamma - 1)C_L^{\frac{\gamma}{\gamma - 1}})$. 
    Since $s\in[0, \delta_0)$ and $C_L > 0$ are finite, we have $\lim_{k \to -\infty} W_k(s) = -\infty$, which contradicts the assumption of non-negative $W_k$. 

    If $0 < a_1 \leq 1$, we rewrite the inequality as
    \begin{equation*}
        \begin{aligned}
            \frac{\sigma^2}{2} W_k' 
            % &=
            % -(\gamma - 1)(1 - x)^{\frac{\gamma}{\gamma - 1}} F(W_k(x)) 
            % - [p_0-(1-x)c_v]a_0 p_0^{-a_1}x^{a_2}
            % + \delta x W_k(x) + \alpha \int_0^x W_k(u) du + \alpha K_k
            % \\
            &\leq \delta C_L x  + \alpha C_L x + \alpha K_k - G(x, a_1), \\
            &\leq \delta C_L x  + \alpha C_L x + \alpha K_k - [p_0-(1-x)c_v]a_0 p_0^{-a_1}x^{a_2},
            \\
            &\leq \delta C_L x  + \alpha C_L x + \alpha K_k + c_v a_0 p^{-a_1} x^{a_2},  
        \end{aligned}
    \end{equation*}
   since $a_0 p_0^{1-a_1} x^{a_2} \geq 0$ and $p_0^{-a_1}x^{a_2 + 1} \geq 0$. 
    Integrating both sides of the inequality over $[0, s]$ to obtain
    \begin{equation*}
    % \begin{aligned}
        W_k(s) 
        % &\leq C_L + \frac{2}{\sigma^2} \left[\delta C_L \frac{s^2}{2} + \alpha C_L \frac{s^2}{2} + \alpha K_k s + c_v a_0 p^{-a_1} \left( \frac{1}{a_2 + 1}s^{a_2 + 1} - \frac{1}{a_2 + 2} s^{a_2 + 2}\right)\right] 
        % \\
        \leq C_L + \frac{2}{\sigma^2} \left(\delta C_L \frac{s^2}{2} + \alpha C_L \frac{s^2}{2} + \alpha K_k s + c_v a_0 p^{-a_1}  \frac{1}{a_2 + 1}s^{a_2 + 1}\right). 
    % \end{aligned}
    \end{equation*}
    If we let $k \to-\infty$, one observe $\lim_{k\rightarrow -\infty} W_k(s) = -\infty$, which is also a contradiction. 
    Consequently, there exists $k_1 < 0$ small such that whenever $k < k_1$, $W_k(x) = 0$ for some $x$, i.e., it is possible for the solution to cross the $x$ axis. This completes the proof. 
\end{proof}

\begin{lemma}
    Consider the sequence of solutions $\{W_k\}_{k \in \mathbb{R}}$ to \eqref{eq: W_k second order}. 
    There exists $k_2 > 0$ large enough such that for each $k > k_2$, we have $W_k(x) >0$, and it has a local maximum. 
    \label{lemma: large k}
\end{lemma}

\begin{proof}
    Since $W_k(0) = C_L$ and $W_k'(0) = k > 0$, we know that $W_k$ is strictly increasing and $W_k(x)\geq C_L$ for $x$ in the neighborhood of the origin. 
    By \eqref{eq: function G prime} and Remark \ref{remark: parameter a_1 <= 1}, we observe that $G'(\cdot, a_1)\geq 0$. 
    % \textcolor{red}{Here, it is straightforward to justify when $a_1 > 1$. When $0 < a_1 \leq 1$, since it is assumed that the unit production cost of virgin resources does not exceed the greatest retail price, namely, $p_0 \geq c_v$, we can easily deduce $G'(\cdot, a_1) \geq 0$. }
    % \textcolor{green}{(how about $p_0 < c_v$? 06/29)}
    Therefore, we may derive an inequality in the neighborhood of the origin as follows: 
    \begin{equation*}
        \frac{\sigma^2}{2} W_k'' \leq \gamma(1 - x)^{\frac{1}{\gamma - 1}}F(W_k) - (\gamma - 1)(1 - x)^{\frac{\gamma}{\gamma - 1}}F'(W_k) W_k'
        + \delta x W_k' + (\alpha + \delta) W_k. 
    \end{equation*}
    If we evaluate both sides of the inequality at $x=0$, we obtain 
    \[
    \frac{\sigma^2}{2}W_k''(0) \leq \gamma C_L^{\frac{\gamma}{\gamma - 1}} - \gamma k C_L^{\frac{1}{\gamma - 1}} + (\alpha + \delta) C_L, 
    \]
    which further yields $\lim_{k\to\infty} W_k''(0) = -\infty < 0$ , which implies a local maximum at the origin. 

    % Next, we intend to exhibit the behavior around the endpoint $x=1$. 
    % Suppose we have $W_k(x) \leq 0$ for $x\in U := \{x\in [0, 1]: 1 - \delta_0 < x < 1\}$ and for large $k$, where $\delta_0 > 0$ is arbitrary. 
    % Note that since $W_k$ is continuous on a closed interval $[0, 1]$, we can find a lower bound, namely, $M_l \leq W_k$, where $M_l < 0$. 
    % Then, \eqref{eq: W_k with integral} implies that for $x\in U$, 
    % \begin{equation*}
    % \begin{aligned}
    %     \frac{\sigma^2}{2} W_k' &= \delta x W_k - G(x, a_1) + \alpha \int_0^x W_k + \alpha K_k \\
    %     &\geq \delta (1-\delta_0) M_l - G(x, a_1) + \alpha M_l (1 - \delta_0) + \alpha K_k. 
    % \end{aligned}
    % \end{equation*}
    % By letting $k\to\infty$ on both sides, we have $\lim_{k\to\infty}W_k'(x) = \infty > 0$ for $x\in U$. 
    % One can find an $\eta\in(0, 1)$ such that $W_k(\eta) < 0$ and $W_k'(\eta) = 0$. 
    % Consider \eqref{eq: W_k with integral} at $x=\eta$. 
    % We observe 
    % \begin{equation*}
    % \begin{aligned}  
    %     \frac{\sigma^2}{2}W_k'(\eta) &= \delta \eta W_k(\eta) - G(\eta, a_1) + \alpha \int_0^\eta W_k + \alpha K_k 
    %     \\
    %     &\geq \delta \eta M_l - G(\eta, a_1) + \alpha \eta M_l + \alpha K_k. 
    % \end{aligned}
    % \end{equation*}
    % Let $k\to\infty$, we obtain $\lim_{k\to\infty} W_k'(\eta) = +\infty > 0$, which is a contradiction. Therefore, such a negative local minimum does not exist for large $k$. 
    % As a consequence, $\lim_{k\to\infty}W_k(x) > 0$ for $x\in U$. 

    % \textcolor{red}{
    Next, we intend to exhibit $W_k > 0$ for arbitrary large $k>0$. 
    This amounts to show that $\lim_{k\to\infty}c_k > 1$. 
    Firstly we assume $\lim_{k\to\infty}c_k < 1$ and hence, there is an $\eta\in(c_k, 1]$ so that $W_k(x) \leq 0$ for $x\in(c_k, \eta]$ and large $k$. 
    Note that since $W_k$ is continuous on a closed interval $[0, 1]$, we have $\|W_k\|_{T} \leq M_u$, where $T=1$ and $M_u > 0$ is a constant. In particular, we can find a lower bound, namely, $M_l \leq W_k$, where $M_l < 0$. 
    Then, \eqref{eq: W_k with integral} for $x\in [c_k, \eta]$ suggests
    \begin{equation*}
    \begin{aligned}
        \frac{\sigma^2}{2}W_k'(x) &= \delta x W_k(x) - G(x, a_1) + \alpha \int_0^x W_k + \alpha K_k 
        \\
        &\geq \delta x M_l - G(x, a_1) + \alpha x M_l + \alpha K_k. 
    \end{aligned}
    \end{equation*}
    If we integrating both sides of the inequality over $(c_k, \eta)$, in the case of $a_1 > 1$ one can deduce
    \begin{equation*}
        \begin{aligned}
            W_k(\eta) &\geq \frac{2}{\sigma^2} \left[ \frac{\delta}{2} M_l(\eta^2 - c_k^2)
            - \int_{c_k}^\eta G(x, a_1)dx 
            + \frac{\alpha}{2} M_l(\eta^2 - c_k^2) 
            + \alpha K_k (\eta - c_k) 
            \right]
            \\
            &\geq \frac{2}{\sigma^2} \left[ \frac{\delta}{2} M_l(\eta^2 - c_k^2)
            +\frac{c}{2-a_1} \left((1-\eta)^{2-a_1} - (1-c_k)^{2-a_1}\right)
            + \frac{\alpha}{2} M_l(\eta^2 - c_k^2)
            + \alpha K_k (\eta - c_k)
            \right], 
        \end{aligned}
    \end{equation*}
    where the last inequality is obtained by employing $(1-x)^{1-a_1}x^{a_2} \leq (1-x)^{1-a_1}$ and compute the integration. 
    In the case of $0 < a_1 \leq 1$, we have 
    \begin{equation*}
        \begin{aligned}
            W_k(\eta) 
            % &\geq \frac{2}{\sigma^2} \left[ \delta M_l (s - c_k) 
            % - \int_{c_k}^s G(\eta, a_1)d\eta 
            % + \alpha M_l(s - c_k) 
            % + \alpha K_k (s - c_k) 
            % \right]
            % \\
            &\geq \frac{2}{\sigma^2} \left[ \frac{\delta}{2} M_l(\eta^2 - c_k^2)
            - \int_{c_k}^\eta \left(p_0 - c_v + c_v x\right) a_0 p_0^{-a_1} x^{a_2}dx
            + \frac{\alpha}{2} M_l(\eta^2 - c_k^2) 
            + \alpha K_k (\eta - c_k)
            \right]
            \\
            &\geq \frac{2}{\sigma^2} \bigg[ \frac{\delta}{2} M_l(\eta^2 - c_k^2)
            - \left(\frac{a_0p_0^{1-a_1}}{a_2 + 1} (s^{a_2+1} - c_k^{a_2+1}) + \frac{a_0 c_v p_0^{-a_1}}{a_2 + 2}(s^{a_2 + 2} - c_k^{a_2 + 2}) \right) \\
            &\quad \quad\quad + \frac{\alpha}{2} M_l(\eta^2 - c_k^2) 
            + \alpha K_k (\eta - c_k)
            \bigg]. 
        \end{aligned}
    \end{equation*}
    Let $k\to\infty$. We observe in both cases that $\lim_{k\to\infty}W_k(\eta) = \infty > 0$, which is a contradiction. 

    It is left to show $\lim_{k\to\infty}c_k \neq 1$. We intend to show this by contradiction. 
    Assume $\lim_{k\to\infty}c_k = 1$ and hence, $W_k(1) = 0$ and $W_k'(1) < 0$ by Lemma \ref{lemma: touch x axis tangentially}. 
    If $0 < a_1 \leq 1$, by \eqref{eq: function G} and \eqref{eq: W_k with integral}, we have
    \[
    \frac{\sigma^2}{2}W_k'(1) \geq - a_0 p_0^{1-a_1} + \alpha K_k. 
    \]
    Let $k\to\infty$, we have $\lim_{k\to\infty}W_k'(1) = \infty > 0$, which is a contradiction. 
    If $a_1 > 1$, we consider the behavior of $W_k$ in a neighborhood of the endpoint, i.e., $U = \{x \in[0, 1]: 1 - \delta_0 < x < 1\}$. 
    The behavior at the endpoint together with $W_k(0) = C_L > 0$ and $W_k'(0) = k > 0$ indicates that we can find an $\eta \in U$ so that $W_k(\eta) > 0$ and $W_k'(\eta) < 0$. 
    By \eqref{eq: W_k with integral} and the boundedness, we have the following inequality: 
    \[
    \frac{\sigma^2}{2} W_k'(\eta) \geq 
    -(\gamma - 1) (1 - \eta)^{\frac{\gamma}{\gamma - 1}} M_u^{\frac{\gamma}{\gamma - 1}}
    -c(1-\eta)^{1-a_1} \eta^{a_2} + \alpha K_k, 
    \]
    where the constant $M_u > 0$ is an upper bound of $W_k$ on the interval $[0, 1]$. 
    Now let $k\to\infty$, it is evident that $\lim_{k\to\infty} W_k'(\eta) =\infty > 0$, which contradicts our assumption. 
    Consequently, we have $\lim_{k\to\infty} c_k > 1$, and hence, $\lim_{k\to\infty}W_k(\cdot) > 0$. 
    % }

    Additionally, since $W_k(x)$ is continuous in $x$ and initial condition $k$ (see Chapter V in \cite{hartman2002ordinary} or Chapter 2 in \cite{teschl2012ordinary}), and if $W_k'(0) = k>0$, we have $W_k(\cdot)$ is strictly increasing in an interval $[0, \delta_k)$, where $\delta_k > 0$ is small and depends on $k$. 
    As a consequence, we can find $k_2 > 0$ large such that $W_{k_2}$ has a local maximum. 
    By comparison lemma (Lemma \ref{lemma: comparison lemma}), there exists $k_2>0$ large such that for each $k>k_2$, $W_k > 0$ and it has a local maximum. 
    This completes the proof. 
\end{proof}

Based on Lemmas~\ref{lemma: comparison lemma}-\ref{lemma: large k}, we obtain a comprehensive understanding of the solution profile of the second-order differential equation~\eqref{eq: W_k second order}. In Figure~\ref{fig: W_k with different k values}, we present sample graphs of $W_k$ corresponding to various values of $k$, where the parameters $k_i$ are ordered increasingly for $i = 0, 1, \cdots, 6$, and satisfy $k_2 < k^* < k_4$. Here, $k^*$ denotes the optimal value such that $W_{k^*}$ is the optimal solution to Equation~\eqref{eq: W_k second order}.
We assume the following parameter values: initial condition $C_L = 0.5$, $\sigma^2 = 2$, $\gamma = 5$, $a_1 = 1.1$, $a_2 = 5$, $\delta = 0.5$, and $c_v = 0.2$. For the performance measure defined in \eqref{equwithbm}, we set the discount parameter $\alpha = \frac{1}{4}$.

\begin{figure}[h!]
    \centering
    \includegraphics[scale=0.7]{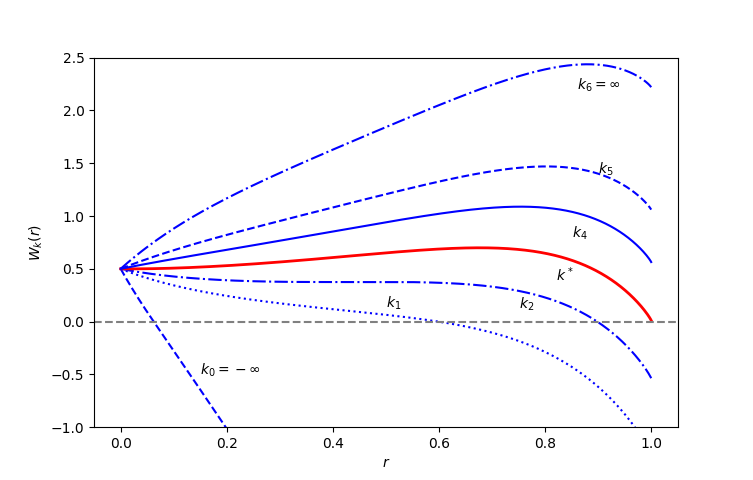}
    \caption{The sample solutions $W_k$ with different values of $k$.}
    \label{fig: W_k with different k values}
\end{figure}

\begin{proposition}
    There exists a $k^*$ such that the corresponding function $Y_{k^*}(x)$  is increasing on $[0, 1]$, and it has a non-negative bounded first derivative such that $0 \leq Y_{k^*}'(x) \leq M$ for all $x\in[0, 1]$ where $M > 0$ large enough. Moreover, we have $Y_{k^*}'(1) = 0$. 
    \label{prop: solution profile for Y_k^*}
\end{proposition}

\begin{proof}
    We introduce
    \begin{equation}
        k^* = \inf\{k\in \mathbb{R}: W_k \text{ has a local maximum and } W_k(1) \geq 0 \},
        \label{eq: k^*}
    \end{equation}
    which is well defined since the set $\{k\in\mathbb{R}: W_k \text{ has a local maximum and } W_k(1) \geq 0 \}$ is nonempty and $k^*$ is finite according to Lemma \ref{lemma: small k} and \ref{lemma: large k}. 
    First, we show that $Y_{k^*}'(1)=W_{k^*}(1) = 0$. 
    We intend to exhibit this by contradiction. 
    Suppose we have $W_{k^*}(1) > 0$. By the continuity of the $W_k$ with respect to $k$ and comparison lemma (Lemma \ref{lemma: comparison lemma}), there exists $\delta_0>0$ such that whenever $-\delta_0 < k - k^* < 0$, we have $W_k < W_{k^*}$, $W_k(x)$ is close to $W_{k^*}(x)$ for all $x$, and $W_k(1) \geq 0$.
    This contradicts the definition of $k^*$. Therefore, $W_{k^*}(1) = 0$. 

    Second, we intend to show $c_{k^*} = 1$ by contradiction. Suppose $c_{k^*} < 1$. Lemma \ref{lemma: touch x axis tangentially} yields $\lim_{x\to c_{k^*}^-} W_{k^*}'(x) < 0$. 
    Since $W_{k^*}(1) = 0$, one can find $\eta\in(c_{k^*}, 1)$ such that $W_{k^*}$ takes a negative local minimum at $\eta$. 
    By \eqref{eq: W_k second order} at $x=\eta$, we observe that 
    \[
        \frac{\sigma^2}{2}W_{k^*}''(\eta) = - G'(\eta, a_1)
        + (\alpha + \delta)W_{k^*}(\eta). 
    \]
    Since $G'(\eta, a_1) \geq 0$, it is straightforward that $W_{k^*}''(\eta) < 0$, which contradicts our assumption. Hence, $W_{k^*}$ reaches $x$ axis at the endpoint for the first time.

    Third, we intend to exhibit that $W_{k^*}(x)$ is concave around $x=1$. Since $W_{k^*}(0) = C_L$, $W_{k^*}(1) = 0$, and $W_{k^*}'(0) = k^*$, \eqref{eq: W_k second order} suggests that for $x\in U := \{x: 1 - \delta_0 < x< 1\}$, where $\delta_0 > 0$ is arbitrary, we have 
    \[
        \frac{\sigma^2}{2}W_{k^*}'' \leq \gamma(1 - x)^{\frac{1}{\gamma - 1}}F(W_{k^*}) 
        - (\gamma - 1)(1 - x)^{\frac{\gamma}{\gamma - 1}} F'(W_{k^*}) W_{k^*}'
        -G'(x, a_1) + (\alpha + \delta) W_{k^*},  
    \]
    since $W_{k^*}'(1) < 0$ by Lemma \ref{lemma: touch x axis tangentially}. 
    By letting $x\to 1$, we obtain $\lim_{x\to 1}W_{k^*}''(x) \leq - \lim_{x\to 1} G'(x, a_1) = -\infty < 0$ if $a_1 > 1$, and $W_{k^*}''(1) \leq - a_0p_0^{-a_1}(c_v + a_2 p_0) < 0$ if $0< a_1 \leq 1$. 
    Thus, $W_{k^*}$ is strictly concave for $x\in U$. 

    With these observations in hand, we can deduce the solution profile of the corresponding $Y_{k^*}$. 
    Since $Y_{k^*}' = W_{k^*}$, it is straightforward deriving that $Y_{k^*}(x) = \int_0^x W_{k^*}(s)ds + K_{k^*}$, where $K_{k^*} = \frac{1}{\alpha} (\frac{\sigma^2}{2} k^* + (\gamma - 1) C_L^{\frac{\gamma}{\gamma - 1}})$. 
    Moreover, $Y_{k^*}'(1) = W_{k^*}(1) = 0$. 
    By the definition of $k^*$ in \eqref{eq: k^*} and in conjunction with the continuity and Lemma \ref{lemma: touch x axis tangentially}, 
    % and no negative oscillation, 
    we have $W_{k^*}$ is bounded on a closed interval $[0, 1]$ and $W_{k^*} \geq 0$, which further yields that $Y_{k^*}'$ is non-negative and bounded. 
    As a consequence, $Y_{k^*}(x)$ is monotonically increasing over $x \in [0, 1]$. 
    This completes the proof. 
\end{proof}

% \begin{remark}
%     In the case of $0 < a_1\leq 1$ and $p_0 < c_v$, we could still observe a local maximum that is greater than $\frac{a_2(c_v - p_0)}{(a_2 + 1)c_v}$. 
%     Suppose there is $\xi\in[\frac{a_2(c_v - p_0)}{(a_2 + 1)c_v}, 1]$ such that $W_k(\xi) > 0$ and $W_k'(\xi) =0$. 
%     Since for $\xi > \frac{a_2(c_v - p_0)}{(a_2 + 1)c_v}$, we have $c_v + a_2 \xi^{-1} [p_0 - c_v(1-\xi)] > 0$, which implies $G'(x, a_1) > 0$. 
%     By \eqref{eq: W_k second order}, we have 
%     \[
%     \frac{\sigma^2}{2}W_k''(\xi) = \gamma (1 - \xi)^{\frac{1}{\gamma - 1}}F(W_k(\xi)) 
%     -G'(x, a_1)
%     +(\alpha + \delta)W_k(\xi)
%     \]
    
%     \textcolor{red}{need more work to show (1) there could be negative local minimum $<\frac{a_2(c_v - p_0)}{(a_2 + 1)c_v}$, (2) the local maximum $>\frac{a_2(c_v - p_0)}{(a_2 + 1)c_v}$ should be positive}
% \end{remark}

Now, we exhibit the proof of Proposition \ref{prop: solution to HJB}. 

\begin{proof}[Proof of Proposition \ref{prop: solution to HJB}]
    For all $x\in[0, 1]$, let $Q(x) = Y_{k^*}(x)$. With the help of Proposition \ref{prop: solution profile for Y_k^*}, it is straightforward that $Q(\cdot)$ satisfies all the assertions. Moreover, since the monotonicity and continuity of $Q(\cdot)$, and $Q(0) = Y_{k^*}(0) = K_{k^*}$, where $K_{k^*} = \frac{1}{\alpha} (\frac{\sigma^2}{2} k^* + (\gamma - 1) C_L^{\frac{\gamma}{\gamma - 1}})$ is a finite constant, we conclude that $Q(\cdot)$ is bounded. 
\end{proof}

\newpage
\bibliography{refs}
\end{document}